\documentclass[12pt]{amsart}
\usepackage[all]{xy}						
\usepackage{amssymb}
\usepackage{amscd,latexsym,amsthm,amsfonts,amssymb,amsmath,amsxtra}
\usepackage[mathscr]{eucal}
\usepackage[colorlinks]{hyperref}
\hypersetup{linkcolor=black, citecolor=black}
\usepackage{pdfsync}

\renewcommand\theequation{\thesection.\arabic{equation}}

\theoremstyle{break}
\newtheorem{thm}{Theorem}[section]
\newtheorem{cor}[thm]{Corollary}
\newtheorem{prop}[thm]{Proposition}
\newtheorem{lemma}[thm]{Lemma}

\newtheorem{rk}[thm]{Remark}

\newtheorem*{Borel's Lemma}{Borel's Lemma}

\newtheorem*{thm*}{Theorem}
\newtheorem{defn}[thm]{Definition}

\newenvironment{proof-idea}{\noindent{\bf Proof Idea}\hspace*{1em}}{\qed\bigskip\\}
\newenvironment{proof-of-lemma}[1]{\noindent{\bf Proof of Lemma #1}\hspace*{1em}}{\qed\bigskip\\}
\newenvironment{proof-attempt}{\noindent{\bf Proof Attempt}\hspace*{1em}}{\qed\bigskip\\}

\newcommand{\sbst}{\subseteq}
\newcommand{\norm}[1]{\lVert#1\rVert}
\newcommand{\abs}[1]{\lvert#1\rvert}
\newcommand{\set}[2]{\{#1\,|\,#2\}}

\newcommand{\mtrtwo}[4]{\begin{pmatrix} #1 &#2 \\#3 &#4 \end{pmatrix}}
\newcommand{\mtrthr}[9]{\begin{pmatrix} #1 &#2 &#3 \\#4 &#5 &#6\\ #7 &#8 &#9 \end{pmatrix}}

\newcommand{\BA}{{\mathbb {A}}}

\newcommand{\BC}{{\mathbb {C}}}

\newcommand{\BR}{{\mathbb {R}}}

\newcommand{\BZ}{{\mathbb {Z}}}

\newcommand{\CZ}{{\mathcal {Z}}}

\newcommand{\RH}{{\mathrm {H}}}
\newcommand{\RI}{{\mathrm {I}}}

\newcommand{\RO}{{\mathrm {O}}}

\newcommand{\RU}{{\mathrm {U}}}

\newcommand{\fk}{\mathfrak{k}}

\newcommand{\ft}{\mathfrak{t}}

\newcommand{\Ad}{\mathrm{Ad}}

\newcommand{\even}{\mathrm{even}}

\newcommand{\re}{\mathrm{Re}}
\newcommand{\rt}{\mathrm{right}}
\newcommand{\lt}{\mathrm{left}}

\newcommand{\R}{\mathbb{R}}
\newcommand{\C}{\mathbb{C}}

\newcommand{\eps}{\epsilon}

\newcommand{\quo}{\backslash}

\renewcommand{\bar}{\overline}
\renewcommand{\tilde}{\widetilde}
\renewcommand{\hat}{\widehat}

\newcommand{\diag}{{\mathrm{diag}}}
\newcommand{\fin}{{\mathrm{fin}}}
\newcommand{\GL}{{\mathrm{GL}}}
\newcommand{\Hom}{{\mathrm{Hom}}}
\newcommand{\id}{{\mathrm{id}}}
\newcommand{\Ind}{{\mathrm{Ind}}}
\newcommand{\Lie}{{\mathrm{Lie}}}
\newcommand{\Mat}{{\mathrm{Mat}}}

\newcommand{\sgn}{{\mathrm{sgn}}}

\newcommand{\SO}{{\mathrm{SO}}}
\newcommand{\SU}{{\mathrm{SU}}}
\newcommand{\Sp}{{\mathrm{Sp}}}

\newcommand{\tr}{{\mathrm{Tr}}}

\def\bks{{\backslash}}

\makeatletter

\newcommand{\Rmnum}[1]{\expandafter\@slowromancap\romannumeral #1@}
\makeatother

\begin{document}
\renewcommand{\theequation}{\arabic{equation}}
\numberwithin{equation}{section}

\title[Explicit Cohomological Test Vectors for $\GL_{2n}(\R)$]{Archimedean Non-vanishing, Cohomological Test Vectors, and Standard $L$-functions of $\GL_{2n}$: Real Case}

\author{Cheng Chen}
\address{School of Mathematics\\
University of Minnesota, USA}
\email{chen5968@umn.edu}

\author{Dihua Jiang}
\address{School of Mathematics\\
University of Minnesota, USA}
\email{dhjiang@math.umn.edu}

\author{Bingchen Lin}
\address{School of Mathematics\\
Sichuan University, China}
\email{87928335@qq.com}

\author{Fangyang Tian}
\address{School of Mathematics\\
National University of Singapore, Singapore}
\email{mattf@nus.edu.sg}


\date{\today}


\thanks{The research of Jiang is supported in part by the NSF Grants DMS--1600685 and DMS--1901802; that of Lin is supported in part by the China Scholarship Council No.201706245006; and
that of Tian is is supported in part by AcRF Tier 1 grant R-146-000-277-114 of National University of Singapore.
}


\begin{abstract}
The standard $L$-functions of $\GL_{2n}$ expressed in terms of the Friedberg-Jacquet global zeta integrals have better structure for arithmetic applications, due to the relation of the linear periods with the
modular symbols. The most technical obstacles towards such arithmetic applications are (1) {\sl non-vanishing of modular symbols at infinity} and (2)
the {\sl existence or construction of uniform cohomological test vectors}. Problem (1) is also called the {\sl non-vanishing hypothesis at infinity}, which was proved by B. Sun in \cite[Theorem 5.1]{Sun},
by establishing the existence of certain cohomological test vectors.

In this paper, we explicitly construct an archimedean local integral that produces a new type of a twisted linear functional $\Lambda_{s,\chi}$, which, when evaluated with our explicitly constructed cohomological vector,
is equal to the local twisted standard $L$-function $L(s,\pi\otimes\chi)$ for all complex values $s$. With the relations between linear models and Shalika models, we establish (1) with an explicitly constructed cohomological vector using classical invariant theory,
and hence proves the non-vanishing results of Sun in \cite[Theorem 5.1]{Sun} via a completely different method.

\end{abstract}

\maketitle


\tableofcontents


\section{Introduction}\label{Section: Introduction}

Let $k$ be a number field, and $\BA$ be the ring of adeles of $k$.
Let $\pi$ be an irreducible cuspidal automorphic representation of the general linear group $\GL_{2n}(\BA)$. The standard $L$-function $L(s,\pi\otimes\chi)$ of $\pi$, twisted by an
idele-class character $\chi$ of $k^\times$, was first studied by R. Godement and H. Jacquet in 1972 (\cite{GJ72}), and then by the Rankin-Selberg convolution method of Jacquet, I. Piatetski-Shapiro and J. Shalika in 1983 (\cite{JPSS83}).
In 1993, S. Friedberg and Jacquet found in \cite{F-J} a new global zeta integral for $L(s,\pi\otimes\chi)$, assuming that $\pi$ has a non-zero Shalika period.

Let $\omega=\omega_\pi$ be the central character of $\pi$ and take an idele-class character $\eta$ such that $\eta^n\cdot\omega=1$. The global zeta integral
$\CZ(\varphi_\pi,\chi,\eta,s)$ of Friedberg-Jacquet is given by
\begin{equation}\label{gzi-FJ}
\int_{[\GL_n\times \GL_n]}\varphi_\pi(\begin{pmatrix}g_1&0\\0&g_2\end{pmatrix})|\frac{\det g_1}{\det g_2}|^{s-\frac{1}{2}}\chi(\frac{\det g_1}{\det g_2})
\eta(\det g_2)dg_1dg_2,
\end{equation}
where $[\GL_n\times \GL_n]:=Z_{2n}(\BA)(\GL_n(k)\times \GL_n(k))\bks (\GL_n(\BA)\times \GL_n(\BA))$, with $Z_{2n}$ the center of $\GL_{2n}$.
In \cite[Proposition 2.3]{F-J}, it is proved that $\CZ(\varphi_\pi,\chi,\eta,s)$ converges absolutely for all $s\in\BC$ and for $\re(s)$ sufficiently large, it is equal to the absolutely convergent integral
\begin{equation}\label{gzi-FJ-S}
\CZ(V_{\varphi_\pi},\chi,s)
:=
\int_{\GL_n(\BA)}V_{\varphi_\pi}\begin{pmatrix}g&0\\0&\RI_n\end{pmatrix}\chi(\det g)|\det g|^{s-\frac{1}{2}}dg
\end{equation}
where $V_{\varphi_\pi}$ is the global Shalika period of $\varphi_\pi$ that is defined as follows. Let $S$ be the Shalika subgroup of $\GL_{2n}$ consisting of matrices of the form
$$
s(x,g)=\begin{pmatrix}\RI_n&x\\0&\RI_n\end{pmatrix}\begin{pmatrix}g&0\\0&g\end{pmatrix}
$$
where $x\in\Mat_n$ and $g\in \GL_n$. Define $\theta_\eta(s(x,g)):=\eta(\det g)\psi^{-1}(\tr(x))$ with a non-trivial additive character $\psi$ of $k\bks \BA$.
The Shalika period of $\varphi_\pi$ is defined by
$$
V_{\varphi_\pi}(h)
:=
\int_{Z_{2n}(\BA)S(k)\bks S(\BA)}\varphi_\pi(s(x,g)h)\theta_\eta(s(x,g))ds.
$$
By the local uniqueness of the Shalika model (\cite{Nien}, \cite{A-G-J}, and \cite{Ch-Sun}), for the factorizable $\varphi_\pi=\otimes_v \varphi_v$, one has that $V_{\varphi_\pi}(h)=\prod_vV_{\varphi_v}(h_v)$ with $V_{\varphi_v}(h_v)$ being the local Shalika function associated to the local Shalika model at each place $v$, and
an euler product decomposition:
$$
\CZ(V_{\varphi_\pi},\chi,s)=\prod_v\CZ_v(V_{\varphi_v},\chi_v,s)
$$
where the local zeta integrals are defined by
\begin{equation}\label{lzi-FJ}
\CZ_v(V_{\varphi_v},\chi_v,s)
:=
\int_{\GL_n(k_v)}V_{\varphi_v}\begin{pmatrix}g&0\\0&\RI_n\end{pmatrix}\chi_v(\det g)|\det g|_v^{s-\frac{1}{2}}dg.
\end{equation}
Furthermore, it is proved that the local zeta integral $\CZ_v(V_{\varphi_v},\chi_v,s)$ is a holomorphic multiple of the local $L$-function
$L(s,\pi_v\otimes\chi_v)$ (\cite{F-J} and \cite{A-G-J}).
It is clear that the Friedberg-Jacquet global zeta integral for $L(s,\pi\otimes\chi)$ is another natural generalization of the classical Hecke-type global zeta integral of Jacquet-Langlands for
$\GL_2$ (\cite{JL70}).

Among the three constructions of different global zeta integrals for $L(s,\pi\otimes\chi)$, it seems that the Friedberg-Jacquet global zeta integral for
$L(s,\pi\otimes\chi)$ is better for applications with $\pi$ being cohomological, since the construction is closely related to the generalized modular symbols (\cite{AB89} and \cite{JST}).
We refer to \cite{JST} for more detailed discussions of the applications of the Friedberg-Jacquet integrals to the period relations of the critical values at different critical places for
the automorphic $L$-functions $L(s,\pi\otimes\chi)$.
For such important applications, it is technically very essential to establish the following properties related to
the Friedberg-Jacquet integral for $L(s,\pi\otimes\chi)$:
\begin{enumerate}
\item {\bf Non-vanishing Hypothesis}: The modular symbol at infinity is non-zero.
\item {\bf Uniform Cohomological Test Vector}: The archimedean local zeta integral $\CZ_v(V_{\varphi_v},\chi_v,s)$ admits a uniform cohomological test vector $\varphi_v$ in the sense that
$$\CZ_v(V_{\varphi_v},\chi_v,s) = L(s, \pi_v\otimes\chi_v)$$
holds as a meromorphic function of $s\in\BC$.
\end{enumerate}
For (1), B. Sun establishes in \cite[Theorem 5.5]{Sun} the non-vanishing hypothesis for real case by showing the existence of certain cohomological test vectors.
For (2), the best result to the date is Sun's existence of cohomological test vector in \cite[Theorem 5.1]{Sun}, which shows that for any irreducible essentially tempered cohomological Casselman-Wallach representation $\pi_{v}$ of $\GL_{2n}(\mathbb{R})$ and every $s\in\BC$, there exists a cohomological vector $\varphi_{v,s}$, depending on $s$, such that the normalized Friedberg-Jacquet integral
$$\frac{1}{L(s, \pi_v\otimes\chi_v)}\CZ_v(V_{\varphi_{v,s}},\chi_v,s) = 1.$$ As explained \cite{JST}, this is not enough to obtain the global period relation of the critical values of
the twisted standard $L$-functions $L(s,\pi\otimes\chi)$ at different critical places.

The objective of this paper is to develop a constructive approach towards Problems (1) and (2), which is complementary to the approach taken by Sun in \cite{Sun}. In this paper, we do the real case, and
leave the complex case to \cite{LT}, which is similar, but has extra complications.

In the process of our understanding of the archimedean local zeta integrals of Friedberg-Jacquet, we find a new type of local integrals, which  produce new linear functionals $\Lambda_{s,\chi}$
and give a different exppression of the linear models for $\GL_{2n}(\BR)$. The explicitly constructed cohomological test vector, which is the most technical work of this paper (Section \ref{Section: Cohomological vector realization}), turns out to be a uniform
cohomological test vector that relates $\Lambda_{s,\chi}$ and the local twisted standard $L$-function $L(s,\pi\otimes\chi)$. With the known relation between the linear models, the Shalika models and the
local zeta integrals of Friedberg-Jacquet, we deduce our main result (Theorem \ref{thm-main}), which recovers the non-vanishing result of Sun
in \cite[Theorem 5.1]{Sun}, with explicitly constructed cohomological test vectors. Meanwhile, Theorem \ref{thm-main} shows that our explicitly constructed cohomological test vectors give a solution to Problem (2)
on uniform cohomological test vector for the archimedean local zeta integral of Friedberg-Jacquet, up to an exponential type function in $s$. This exponential type function in $s$ will be removed with
full details, including the complex case, in \cite{JST}, which will lead to a complete proof of the global period relation of critical values of the twisted standard $L$-functions at
different critical places for irreducible, regular algebraic,
cuspidal automorphic representations of $\GL_{2n}$ of (generalized) symplectic type (\cite{JST}).

\subsection{Cohomological representations of $\GL_{2n}(\BR)$}\label{sec-CRGL2n}
Let $G=\GL_{2n}(\R)$, and $K = \RO_{2n}(\R)$ be the maximal compact subgroup of $G$. Let $B$ be the Borel subgroup of $G$ consisting of
all upper-triangular matrices in $G$. We fix the usual root system of $G$ so that $B$ contains all simple root vectors. Then the half sum of all positive roots, denoted by $\rho$, is
\begin{equation}\label{eq: half sum positive roots}
   \rho = (\frac{2n-1}{2}, \frac{2n-3}{2}, \cdots, \frac{3-2n}{2}, \frac{1-2n}{2}).
\end{equation}
It is clear that all standard parabolic subgroups of $G$ are in one-to-one correspondence with the ordered partition of $2n$. For instance, when $n=2$, we regard $4=1+3$ and $4=3+1$ as different partitions.
Accordingly, they correspond to standard parabolic subgroups of $\GL_4(\BR)$ whose Levi subgroups are $\GL_1(\BR)\times \GL_3(\BR)$ and $\GL_3(\BR)\times\GL_1(\BR)$ respectively.

Set
    \begin{equation*}
       H = \Big\{\mtrtwo{g_1}{}{}{g_2}\Big\lvert g_1,g_2\in \GL_n(\R) \Big\} \simeq \GL_n(\R) \times \GL_n(\R).
    \end{equation*}
    Denoted by $Z$ the center of $G$. Let $d$ be the dimension of the quotient space
    $\Lie(H)/\Lie((K\cap H)Z)$, where $\Lie(H)$ is the Lie algebra of $H$. Then
    \begin{equation}\label{Eq: Def of d}
       d =  n^2+n-1,
    \end{equation}
    which, as suggested by \cite{A-G1} and exhibited in \cite[Section 3.4]{G-R}, is the dimension of the modular symbol generated by the closed subgroup $H$. To fix notation, from now on, we will use capital letters
$G, H$ etc. for certain Lie groups, $G^0, H^0$ etc. for their identity components, German letters $\mathfrak{g}, \mathfrak{h}$ etc. for their Lie algebras, and $\mathfrak{g}^\C, \mathfrak{h}^\C$ for
the complexifications of the Lie algebras.

    Let $F_\nu$ be a highest weight representation of $\GL_{2n}(\C)$ with highest weight $\nu$, which can be written as a vector of the following type:
    \begin{equation}\label{eq: highest weight of f-d rep}
         \nu = (\nu_1,\nu_2,\cdots, \nu_{2n})\in \mathbb{Z}^{2n},
    \end{equation}
    with $\nu_1\geq\nu_2\geq\cdots\geq\nu_{2n}$. The non-vanishing hypothesis involved in the above applications suggests that in this paper we only need to consider the irreducible essentially tempered Casselman-Wallach representations $(\pi, V_\pi)$ of $\GL_{2n}(\R)$ with property that the total relative Lie algebra cohomology
    $$\RH^*(\mathfrak{g},K^0Z^0,\pi\otimes F_\nu^\vee)$$ is non-zero.
    By abuse of notation, we also use $\pi$ for its underlying $(\mathfrak{g}, K)$-module when no confusion arises.

    Now we recall the Langlands parameter for $\pi$. (For a general reference, see \cite[Section 3]{Clo}. Also see \cite[Section 3.1]{M} and \cite[Section 3.4]{G-R}.) If the cohomology group
    \begin{equation}\label{eq: cohomological group}\RH^j(\mathfrak{g},K^0Z^0,\pi\otimes F_\nu^\vee)\end{equation}
    is nonzero, then the highest weight $\nu$ satisfies the following purity condition:
    \begin{equation}\label{eq: purity condition}
         \nu_1+\nu_{2n} = \nu_2+\nu_{2n-1} = \cdots = \nu_{2n}+\nu_{1} := m\in \mathbb{Z}.
    \end{equation}
    The integer $m$ is exactly the integer $w$ in \cite[Section 3.1.4]{M} and \cite[Section 3.4]{G-R}. We save the notation $w$ for Weyl elements. Moreover, we must also have that
    \begin{equation*}
           n^2 \leq j \leq n^2+n-1.
    \end{equation*}
    One should observe that the top non-vanishing degree $n^2+n-1$ is exactly the dimension of the modular symbol $d$ defined in \eqref{Eq: Def of d}. This is the key numerical coincidence,
with which Sun is able to prove the non-vanishing of cohomological maps based on non-vanishing of one archimedean integral.

     Given a highest weight $\nu$ and an integer $m$ satisfying \eqref{eq: purity condition}, we define    \begin{equation}\label{eq: def of l}
       \vec{l} := 2\nu+2\rho-(m,m,\cdots,m),
    \end{equation}
    where $\rho$ is given in \eqref{eq: half sum positive roots}.
    If we write $\vec{l} = (l_1,l_2,\cdots,l_{2n})$, then
    \begin{equation}\label{eq: l_i eq}
       l_i = \nu_i-\nu_{2n+1-i}+(2n+1-2i) \quad\text{ for all }1\leq i\leq 2n.
    \end{equation}
    We note that all $l_i$ share the same parity, which is different from the parity of $m$. For each positive integer $k$,
    we write $D_{k}$ for the relative discrete series of $\GL_2(\R)$ with quadratic central character whose minimal $K$-type has highest weight $k+1$. Denote by $\eta_k$ the central character of $D_k$, which is given by
    \begin{equation}\label{eq: central char discrete series}
         \eta_k = \begin{cases} \text{id} &\text{ if $k$ is odd};\\
                                \text{sgn} &\text{ if $k$ is even}.
                   \end{cases}
    \end{equation}
    The cohomological representation $\pi$ must be isomorphic to the normalized parabolically induced representation
    \begin{equation}\label{Eq: pi parabolic induction parameter}
       \begin{aligned}
       \pi \simeq \Ind_{P}^{G} D_{l_1}\abs{\quad}^{\frac{m}{2}}\otimes D_{l_2}\abs{\quad}^{\frac{m}{2}} \otimes \cdots\otimes D_{l_n}\abs{\quad}^{\frac{m}{2}},
       \end{aligned}
    \end{equation}
    where $P$ is the standard parabolic subgroup of $G$ associated with the partition $[2^n]$. The central character $\omega_\pi$ of $\pi$ takes the form
    \begin{equation}\label{eq: central char}
       \omega_\pi(a \RI_{2n}) = \begin{cases} \abs{a}^{mn} \quad\text{ if $m$ is even};\\
                                       \abs{a}^{mn} (\text{sgn}(a))^n \quad\text{ if $m$ is odd}.\\
       \end{cases}
    \end{equation}
    We define a character $\omega$ of $\R^\times$ as follows:
    \begin{equation}\label{eq: omega char}
       \omega(a) = \begin{cases} \abs{a}^{m} \quad\text{ if $m$ is even};\\
                                       \abs{a}^{m} \text{sgn}(a) \quad\text{ if $m$ is odd}.\end{cases}
    \end{equation}
    Then $\omega$ is exactly the central character of $D_{l_i}\abs{\quad}^{\frac{m}{2}}$, for all $i$ and $\omega_\pi (a\RI_{2n}) = (\omega(a))^n$. Clearly, the restriction of $\omega$ on $\{\pm 1\}$, denoted by $\omega_0$, is trivial when $m$ is even; and is the sign character when $m$ is odd. The highest weight of the minimal $K$-type of $\pi$,
    which is denoted by $\tau$,  is $(l_1+1,l_2+1,\cdots,l_n+1)$. We will explain the highest weight for $K=\RO_{2n}(\R)$ with more details
    in Subsection \ref{subsection: Construction of the Scalar-valued Function in the Minimal K-type}. The above discussion can be summarized in the following proposition.
\begin{prop}\label{structure-pi}
Let $(\pi, V_\pi)$ be an irreducible essentially tempered Casselman-Wallach representation of $G$ with property that
    $$\RH^*(\mathfrak{g},K^0Z^0,\pi\otimes F_\nu^\vee)\ne 0.$$
Then $\pi$ is equivalent to the normalized induced representation
$$
\Ind_{P}^{G} D_{l_1}\abs{\quad}^{\frac{m}{2}}\otimes D_{l_2}\abs{\quad}^{\frac{m}{2}} \otimes \cdots\otimes D_{l_n}\abs{\quad}^{\frac{m}{2}},
$$
as given in \eqref{Eq: pi parabolic induction parameter}, and with the central character $\omega_\pi$ being given in \eqref{eq: central char}.
Moreover, the minimal $K$-type $\tau$ of $\pi$ has the highest weight
$(l_1+1,l_2+1,\cdots,l_n+1)$.
\end{prop}

\subsection{Shalika model and linear model}\label{sec-SMLM}
    Let $F = \R$ or $\C$. Let us fix a non-trivial unitary character $\psi$ of $F$ and a multiplicative character $\tilde{\omega}$ of $F^\times$. For any positive integer $n'$, we say a Casselman-Wallach representation $\sigma$ of $\GL_{2n'}(F)$ has a non-zero $(\tilde{\omega}, \psi)-$Shalika model if there exists a non-zero continuous linear functional $\lambda$ on the Fr\'{e}chet space $V_\sigma$, which is called a {\sl Shalika functional}, such that
    \begin{equation}\label{Eq: Def of Shalika Functional}
        \langle \lambda, \pi(\mtrtwo{\RI_{n'}}{Y}{}{\RI_{n'}}\mtrtwo{g}{}{}{g})v \rangle =\tilde{\omega}(\det g) \psi(\tr(Y))\langle \lambda, v\rangle ,
    \end{equation}
    for any $v\in V_\sigma$, $g\in \GL_{n'}(F)$ and any $n'\times n'$ matrix $Y\in\Mat_{n'}(F)$. In Section \ref{section: existence of Shalika model}, we will show (in Theorem \ref{thm: parabolic induction preserves shalika model})
    that the normalized parabolic induction from two representations with non-zero $(\tilde{\omega}, \psi)-$Shalika models also admits a non-zero $(\tilde{\omega}, \psi)-$Shalika model. The existence of the $(\tilde{\omega}, \psi)-$Shalika model of the parabolic induction should be viewed as a direct archimedean analogue of \cite[Theorem 1.1]{Mat}. As a consequence, if we apply $F=\R$ and $\tilde{\omega} = \omega$, the central character of each $D_{l_j}\abs{\quad}^{\frac{m}{2}}$, we can conclude that the representation $\pi$ as in Proposition \ref{structure-pi} also has a non-zero $(\omega, \psi)-$Shalika model.  In the literature, the uniqueness of Shalika functional is proved in \cite[Theorem 1.1]{A-G-J} when $\tilde{\omega}$ is trivial. For the general case, the same result is confirmed in
    \cite[Theorem A]{Ch-Sun}. For a character $\chi$ of $\R^\times$,  the local archimedean integral of Friedberg-Jacquet at real place as in \eqref{lzi-FJ} can be re-written as
    \begin{equation}\label{Eq: Def of Local Integral}
       Z(v,s,\chi) = \int_{\GL_n(\R)} \langle \lambda, \pi(\mtrtwo{g}{}{}{\RI_n})v \rangle \abs{\det g}^{s-\frac{1}{2}}\chi(\det g) dg
    \end{equation}
    for $v\in V_\pi$. We note that when both $\chi$ and $\omega_\pi$ are trivial, the integral in \eqref{Eq: Def of Local Integral}
    is exactly the local integral considered in \cite{A-G1}.

    By \cite[Theorem 3.1]{A-G-J}, the integral \eqref{Eq: Def of Local Integral} converges absolutely when $\re(s)$ is sufficiently large. $Z(v,s,\chi)$ is a homomorphic multiple of $L(s,\pi\otimes\chi)$ in the sense of meromorphic continuation and there exists a smooth vector $v\in V_\pi$ such that $Z(v,s,\chi)=L(s,\pi\otimes\chi)$. Thus whenever $s=s_0$ is not a pole of $L(s,\pi\otimes\chi)$, $Z(v,s,\chi)$ has no pole at $s=s_0$. This implies that the map: $v \mapsto Z(v, s_0,\chi)$ defines a nonzero element in
    $$
    \text{Hom}_{H}(V_\pi, \abs{\det}^{-s_0+\frac{1}{2}}\chi^{-1}(\det)\otimes\abs{\det}^{s_0-\frac{1}{2}}(\chi\omega)(\det)),
    $$
    which is called the space of {\sl twisted linear functionals} of $\pi$. The uniqueness of the twisted linear model is proved in \cite{Ch-Sun}. In our scenario, we apply \cite[Theorem B]{Ch-Sun} and conclude that for all but countably many characters $\chi$,
    \begin{equation}\label{Eq: uniqueness of twisted linear model}
        \text{dim Hom}_{H}(\pi, \abs{\det}^{-s_0+\frac{1}{2}}\chi^{-1}(\det)\otimes\abs{\det}^{s_0-\frac{1}{2}}(\chi\omega)(\det))\leq 1.
    \end{equation}
    In fact, if
        $\abs{\det}^{-s_0+\frac{1}{2}}\chi^{-1}(\det)\otimes\abs{\det}^{s_0-\frac{1}{2}}(\chi\omega)(\det)$
    is a good character of $H$, then \eqref{Eq: uniqueness of twisted linear model} holds. For the precise definition of a good character of $H$, we refer to \cite{Ch-Sun}. Finally we remark that when $\chi$ and $\omega$ are both trivial and $s_0=\frac{1}{2}$, the uniqueness theorem is proved in \cite{A-G}.

\subsection{Cohomological test vectors and non-vanishing property}\label{sec-MR}
 As explained above, any irreducible essentially tempered Casselman-Wallach representation $\pi$ of $G$ with non-zero cohomology must be of the form given in Proposition \ref{structure-pi}, and
such a $\pi$ must have a non-zero Shalika model. Yet, even with an explicit vector $v\in V_\pi$, the Shalika function $\langle \lambda, \pi(\mtrtwo{g}{}{}{\RI_{n}})v\rangle$ is hard to evaluate, let alone the local  Friedberg-Jacquet integral $Z(v,s,\chi)$. Alternatively, we construct explicitly another twisted linear functional $\Lambda_{s,\chi}$
 given by an explicit integral in \eqref{eq: new def H inv linear functional}. The advantage of this newly constructed twisted linear functional is that the value of $\Lambda_{s,\chi}$ at a vector $f\in V_\pi$ with a certain right $K\cap H$-equivariance (realize $f$ as a function using parabolic induction) is easy to compute once we know the value of $f$ at one certain Weyl element $w$ (see \eqref{eq: 020}).  We recall that when $\pi$ is cohomological, a cohomological vector $v$ of $\pi$ is defined to be a vector $v$ belonging to the minimal $K$-type $\tau$ of $\pi$. Thus, to prove the non-vanishing result of Friedberg-Jacquet integral $Z(v,s,\chi)$, it suffices to construct \textbf{explicitly} a function $f\in V_\pi$ in the minimal $K$-type $\tau$ with a certain right $K\cap H$-equivariance, which is a highly non-trivial task.

 Let us only outline the key ingredients in the explicit construction of the cohomological vector $f\in V_\pi$.
 \begin{enumerate}
   \item With some reductions discussed in Subsection \ref{subsection: reduction to polynomial repn} via the compact induction model, it suffices to construct a $\BC$-valued function, which will be denoted by $\varphi_{\vec{N},\chi_0\otimes\chi_0\omega_0}(x)$, in the minimal $K$-type of a certain induced representation $\pi_{sc}$ defined in \eqref{eq: induced space complex value}.
   \item According to the representation theory of compact Lie groups, we have the fact that under the left $K$-action, if a function $\varphi$ is a lowest or highest weight vector of an irreducible $K$-submodule $\alpha$ of $C^\infty(K)$, then it generates $\alpha^*$ under the right $K$-action. See Lemma \ref{lemma: useful dual lemma}.
   \item Via Lemma \ref{lemma: useful dual lemma} and the classical invariant theory, we can construct a weight-building function $F_{wt, \vec{N}}(x)\in C^\infty(K)$ in Corollary \ref{cor: weight building function minimal K type}, which will help us build up the weight of $\varphi_{\vec{N},\chi_0\otimes\chi_0\omega_0}(x)$. Two determinant functions $F_{\lt}(x), F_{\rt}(x)$ that govern the right $(K\cap H)$-equivariance are constructed in Proposition \ref{prop: det lives in fundamental repn}. Then define $\varphi_{\vec{N},\chi_0\otimes\chi_0\omega_0}(x)$ to be a product of them (see Theorem \ref{thm: phi0 lives in minimal K type}).
 \end{enumerate}

  The explicit construction of the cohomological vector $v=f(k)$ as in Corollary \ref{cor: construction of cohomological vector} based on $\varphi_{\vec{N},\chi_0\otimes\chi_0\omega_0}$ leads to the main theorem of this paper:

\begin{thm}[Main Theorem]\label{thm-main}
Let $(\pi, V_\pi)$ be an irreducible essentially tempered Casselman-Wallach representation of $G$ with property that
     $$\RH^*(\mathfrak{g},K^0Z^0,\pi\otimes F_\nu^\vee)\ne 0,$$
and $\Lambda_{s,\chi}$ be the linear functional on $V_\pi$ defined in \eqref{eq: new def H inv linear functional}.
Then the following hold.
\begin{enumerate}
  \item Whenever $s$ is not a pole of the $L$-function $L(s, \pi\otimes\chi)$, \begin{equation*}
       \Lambda_{s,\chi}\in \Hom_{H}(\pi, \abs{\det}^{-s+\frac{1}{2}}\chi^{-1}(\det)\otimes\abs{\det}^{s-\frac{1}{2}}(\chi\omega)(\det)).
    \end{equation*}
  \item There exists a cohomological vector $v=f(k)$ as explicitly constructed in Corollary \ref{cor: construction of cohomological vector} such that as a meromorphic function of $s\in \BC,$
$$
     \frac{1}{L(s,\pi\otimes\chi)}\Lambda_{s,\chi}(v) = 1.
$$
\end{enumerate}
\end{thm}
    \begin{rk}
     It turns out that the integral $\Lambda_{s,\chi}$ constructed in this paper is exactly the Friedberg-Jacquet integral, up to suitable normalization of Haar measures. The proof will be given in \cite{JST} with full details, together with an application to the global period relations of critical values of the twisted standard $L$-functions at different critical places.
     \end{rk}

 The most technical parts of the paper are: one is to construct the new linear functional $\Lambda_{s,\chi}$ in Section \ref{Subsection: Another Linear Model}; and the other is the explicit construction of the right cohomological vector $v=f(k)$ of $\pi$ belonging to the minimal $K$-type and having the desired property in
Section \ref{Section: Cohomological vector realization}.

Finally, we would like to thank Lei Zhang for helpful conversation during our preparation of this paper, and thank Binyong Sun for sending us his preprints \cite{Ch-Sun} and \cite{Sun}
when we were writing up this paper. We would also like to thank the referee for very helpful comments and suggestions.


\section{Cohomological Representations and Shalika Models}\label{section: existence of Shalika model}


The goal of this Section is to prove a hereditary property of Shalika models with respect to normalized parabolic induction. Let us first start from the case of $\GL_2.$ For future applications, we will consider both the real case and complex case in this Section, i.e. $F=\R $ or $\C$.

Suppose that we have $l$ Frech\'{e}t spaces $V_1, V_2, \cdots, V_l$. Let $\lambda_j$ be a continuous linear functional on $V_j$ ($j =1,2,\cdots,l$). Then $\bigotimes_{j=1}^l\lambda_{j}$ is a continuous linear functional
on the projective tensor space $\widehat{\bigotimes}_{j=1}^l V_{j}$, which is also a Fr\'{e}chet space.

Let $\sigma$ be a generic Casselman-Wallach representation of $\GL_2(F)$ with a central character $\omega_\sigma$. We fix a nontrivial unitary character $\psi$ of $F$. Then $\sigma$ admits a non-zero Whittaker model $\mathcal{W}(\sigma,\psi)$, i.e. there exists a continuous linear functional $\lambda_\sigma$ on the Fr\'{e}chet space $V_\sigma$ such that
    \begin{equation*}
        \langle \lambda_\sigma, \sigma(\mtrtwo{1}{x}{}{1})v \rangle = \psi(x)\langle \lambda_\sigma, v\rangle.
    \end{equation*}
    Thus
    \begin{equation*}
        \langle \lambda_\sigma, \sigma(\mtrtwo{1}{x}{}{1}\mtrtwo{a}{}{}{a})v \rangle = \omega_\sigma(a)\psi(x)\langle \lambda_\sigma, v\rangle,
    \end{equation*}
    which exactly coincides with \eqref{Eq: Def of Shalika Functional} for $n=1$. Hence any such $\sigma$ has a non-zero $(\omega_\sigma, \psi)-$Shalika model.

    The main theorem in this Section is formulated below, which can be regarded as an archimedean analogue of \cite[Theorem 1.1]{Mat}.
    \begin{thm}\label{thm: parabolic induction preserves shalika model}
        Let $F = \R$ or $\C$ be an archimedean local field. Let $\omega$ be a character of $F^\times$ and $\psi$ be a nontrivial unitary character of $F$. For two even positive integers $n_1 = 2m_1$ and
$n_2 = 2m_2$, take two Casselman-Wallach representations $\pi_1$ and $\pi_2$ of $\GL_{n_1}(F)$ and $\GL_{n_2}(F)$, respectively, and
assume that both $\pi_1$ and $\pi_2$ have $(\omega, \psi)-$Shalika models. Then the normalized parabolic induction
        $\pi := \Ind_{P_{n_1,n_2}(F)}^{\GL_{n_1+n_2}(F)} \pi_1\otimes\pi_2$
        also has a non-zero $(\omega, \psi)-$Shalika model. Here $P_{n_1,n_2}$ is a standard parabolic subgroup of $\GL_{n_1+n_2}$ with its Levi part isomorphic to $\GL_{n_1}\times\GL_{n_2}$.
    \end{thm}
    The proof of Theorem \ref{thm: parabolic induction preserves shalika model} will occupy the rest of this Section. Our proof is very similar to that of \cite[Theorem 1.1]{Mat}, and we borrow some continuity arguments from \cite[Section 6.3]{JST}. We start with some estimates on Shalika functionals.
    \begin{lemma}\label{lemma: enhanced estimates shalika}
        Let $k$ be a positive integer and $\sigma$ be a Casselman-Wallach representation of $\GL_{2k}(F)$ with a Shalika functional $\lambda_\sigma$. Then there exists a positive integer $M_0$ with property that for any integer $M\geq M_0$ and any polynomial $P$ on $\Mat_k(F)$, there exists a continuous seminorm $\beta_{P,M}$, such that the following estimate
            \begin{equation}\label{eq: enhanced lemma pre}
                 \abs{P(g)}\cdot\abs{\langle \lambda_\sigma, \sigma(\mtrtwo{g}{}{}{\RI_k})v \rangle} \leq \beta_{P,M}(v)\cdot \abs{\det g}^{-M}
            \end{equation}
            holds for all $g\in \GL_k(F)$ and $v\in V_\sigma$.
    \end{lemma}
    \begin{proof}
        In \cite{A-G-J}, the authors consider the Shalika functional when $\omega$ (in \eqref{Eq: Def of Shalika Functional}) is trivial. Their proof of \cite[Lemma 3.4]{A-G-J} is still valid when $\omega$ is non-trivial. Also, we note that in the proof of \cite[Lemma 3.4]{A-G-J}, the authors
        define a norm
        $\norm{\cdot}$ on $\GL_{2k}(F)$ and use the fact that the product $\norm{\mtrtwo{g}{}{}{\RI_k}}^{M_0}\cdot \abs{\det{g}}^{M_0}$ is a polynomial. Hence for any integer $M \geq M_0$, $\norm{\mtrtwo{g}{}{}{\RI_k}}^{M_0}\cdot \abs{\det{g}}^{M}$ is still a polynomial. Then a word-by-word repetition of the proof of \cite[Lemma 3.4]{A-G-J} will confirm the lemma here. We omit the details here.
    \end{proof}
    \begin{cor}\label{cor: enhanced estimate shalika}
        With $k$, $\sigma$, and $\lambda_\sigma$ as given in Lemma \ref{lemma: enhanced estimates shalika},
        there exists an positive integer $M_0$ and a real number $m_0$ with property that for any positive integer $M\geq M_0$,
        there exists a continuous seminorm $\beta_M$ satisfying
        \begin{equation}\label{eq: enhanced lemma}
                 \abs{\langle \lambda_\sigma, \sigma(\mtrtwo{g}{}{}{h})v \rangle} \leq \beta_M(v)\cdot\abs{\det g}^{-M}\abs{\det h}^{M-m_0},
        \end{equation}
        for all $g,h\in \GL_{k}(F)$, and $v\in V_\sigma$.
    \end{cor}
    \begin{proof}
        Let $\omega = \tilde{\omega}$ be the character associated with the given Shalika functional $\lambda_\sigma$ as in \eqref{Eq: Def of Shalika Functional}. Then
        \begin{equation}\label{eq: 017}
              \abs{\langle \lambda_\sigma, \sigma(\mtrtwo{g}{}{}{h})v \rangle}= \abs{\omega(\det h)\langle \lambda_\sigma, \sigma(\mtrtwo{h^{-1}g}{}{}{\RI_k})v \rangle}.
        \end{equation}
        We deduce Corollary \ref{cor: enhanced estimate shalika} by writing $\abs{\omega(\det h)} = \abs{\det h}^{m_0}$ and applying
        Lemma \ref{lemma: enhanced estimates shalika} to the case of $P=1$.
    \end{proof}
    Now we apply such general estimates to the case of Theorem \ref{thm: parabolic induction preserves shalika model} and obtain the following
    estimate.
    \begin{cor}\label{lemma: Shalika functional Schwartz}
        Take $n_1=2m_1$, $n_2=2m_2$, $\pi_1$, $\pi_2$, and $\pi$ as in Theorem \ref{thm: parabolic induction preserves shalika model}. Write $m = m_1+m_2$ and denote by $\lambda_1,\lambda_2$ the Shalika functionals on $V_{\pi_1},V_{\pi_2}$. Write $K_m$ for the maximal compact subgroup of $\GL_m(F)$. There exists a positive integer $M_0$ and a real number $m_0$ such that for any positive integer $M\geq M_0$, the following property holds: there exists a continuous seminorm of $q_M$ on the Fr\'{e}chet space $V_\pi$, depending on $M$, such that for all $f\in V_\pi$, the following estimate
        \begin{eqnarray}\label{eq: estimate of shalika functional}
           &&\abs{\langle \lambda_1\otimes \lambda_2, \pi_1(\mtrtwo{\RI_{m_1}}{}{}{a})\otimes\pi_2(\mtrtwo{b}{}{}{\RI_{m_2}}) f(\mtrthr{\RI_{m_1}}{}{}{}{k}{}{}{}{\RI_{m_2}}) \rangle} \\
           &\leq &\abs{\det a}^{M-m_0} \cdot \abs{\det b}^{-M} \cdot q_M(f)\nonumber
        \end{eqnarray}
        holds for any $a\in \GL_{m_1}(F)$, $b\in \GL_{m_2}(F)$ and $k\in K_m$.
    \end{cor}
    \begin{proof}
        By Corollary \ref{cor: enhanced estimate shalika}, there exists a positive integer $M_0$ with the property that for any positive integer $M\geq M_0$, there exists a continuous seminorm $q_M'$ on $V_{\pi_1}\hat{\otimes} V_{\pi_2}$  such that the following estimate holds for all $v\in V_{\pi_1}\hat{\otimes} V_{\pi_2}$:
        \begin{eqnarray}\label{eq: temp 10}
           &&\abs{\langle \lambda_1\otimes \lambda_2, \pi_1(\mtrtwo{\RI_{m_1}}{}{}{a})\otimes\pi_2(\mtrtwo{b}{}{}{\RI_{m_2}}) v \rangle} \\
           &\leq &\abs{\det a}^{M-m_0} \cdot \abs{\det b}^{-M} \cdot q_M'(v).
        \end{eqnarray}
        Define
        \begin{equation}\label{eq: temp 11}
           q_M(f) := \sup_{k\in K_n} q'_M(f(k)),
        \end{equation}
        where $K_n$ is the standard maximal compact subgroup of $\GL_n(F).$
        We refer the reader to \cite[10.1.1]{WallachRealReductiveGroups2} for the Fr\'{e}chet topology of the parabolically induced representation $V_\pi$.
Under such a Fr\'{e}chet topology, the function $q_M$ defined in \eqref{eq: temp 11} is a continuous seminorm on $V_\pi.$ Thus, the estimate \eqref{eq: estimate of shalika functional} follows directly from \eqref{eq: temp 10}. We are done.
    \end{proof}
    We also need a lemma on the diagonal matrix appearing in the Iwasawa decomposition of a lower unipotent matrix.
    \begin{lemma}\label{lemma: Iwasawa decomposition lower}
        Let $m_1, m_2$ be two positive integers and set $m=m_1+m_2.$ Let $\bar{u}(x):=\mtrtwo{\RI_{m_1}}{}{x}{\RI_{m_2}}$ be an element in the unipotent radical of the lower parabolic subgroup of $\GL_{m}(F)$, and
        \begin{equation*}
             \bar{u}(x) = u(x)t(x)k(x)
        \end{equation*}
        be its Iwasawa decomposition, where
            $t(x) = \diag(t_1(x),t_2(x),\cdots,t_{m}(x))$,
        with all $t_i(x)>0$. Here $u(x)$ lives in the standard maximal unipotent subgroup, $k(x)$ lives in the maximal compact subgroup. Define also
        \begin{equation}\label{eq: def of a(x)}
           \begin{aligned}
                a_1(x) &= \diag(t_1(x),t_2(x),\cdots, t_{m_1}(x)),\\
                a_2(x) &= \diag(t_{m_1+1}(x),t_{m_1+2}(x),\cdots, t_{m}(x)).
           \end{aligned}
        \end{equation}
        Write $x = \left(\begin{array}{c}
                                     x_1 \\
                                     x_2 \\
                                     \vdots \\
                                     x_{m_2} \\
                                   \end{array}
                                 \right)
        $, where each $x_j$ is a row vector in $F^{m_1}$. Denote by $\norm{\cdot}$ the norm induced by the standard inner product in $F^{m_1}$.
        Then the following estimate holds
        \begin{equation}\label{eq: estimate of det a2(x)}
           (1+\norm{x_{1}}^2+\norm{x_{2}}^2+\cdots+\norm{x_{m_2}}^2)^{\frac{1}{2}}\leq \det(a_2(x))\leq \prod_{j=1}^{m_2} (1+\norm{x_j}^2)^{\frac{1}{2}}
        \end{equation}
    \end{lemma}
    \begin{proof}
        The following proof is a standard argument using the exterior algebra. Let $V:= F^{m}$ be the vector space equipped with the standard inner product.
        We also write $\norm{\cdot}$ for the norm induced by this inner product when there is no confusion on the dimension of the vector space.
        Let $e_1$, $e_2, \cdots, e_m$ be the standard orthonormal basis of $V$. Then every vector in $V$ is regarded as a row vector.
        For any positive integer $p$, the vector space $\bigwedge^p V$ of the $p$-th exterior power of $V$ is equipped with a standard inner product $\langle\cdot,\cdot\rangle_p$ defined by
        \begin{equation*}
            \langle v_1\wedge v_2\wedge\cdots\wedge v_p, w_1\wedge w_2\wedge\cdots\wedge w_p \rangle_p := \det(\langle v_i, w_k \rangle).
        \end{equation*}
        Denote by $\norm{\cdot}_{p}$ the norm induced by the inner product $\langle\cdot,\cdot\rangle_p$. Then
        \begin{equation}\label{eq: relation between two norms}
            \norm{v_1\wedge v_2\wedge\cdots\wedge v_p}_p \leq \norm{v_1}\cdot\norm{v_2}\cdots\norm{v_p}
        \end{equation}
        Under the inner product $\langle\cdot,\cdot\rangle_p$, all $e_{i_1}\wedge e_{i_2}\wedge\cdots\wedge e_{i_p}$ $(1\leq i_1<i_2<\cdots<i_p\leq n)$
        form an orthonormal basis of $\bigwedge^p V$. The group $\GL_m(F)$ acts on $V$ by right multiplication, and then induces an action
        on $\bigwedge^p V$. Under such an action, the maximal compact subgroup preserves the inner product $\langle\cdot,\cdot\rangle_p$. We also note that
        for any $1\leq i \leq m$ and any $u$ in the standard maximal unipotent subgroup, one has
        \begin{equation*}
             (e_i\wedge e_{i+1}\wedge\cdots\wedge e_m) \cdot u = e_i\wedge e_{i+1}\wedge\cdots\wedge e_m \cdot
        \end{equation*}
        Thus, using the Iwasawa decomposition $\bar{u}(x) = u(x)t(x)k(x)$, we obtain that
        \begin{equation}\label{eq: bar u(x) wedge 1}
            \norm{(e_i\wedge e_{i+1}\wedge\cdots\wedge e_m)\bar{u}(x)}_p = t_i(x)t_{i+1}(x)\cdots t_m(x).
        \end{equation}
        We also set $x_j' = (x_j,0)\in F^{m}$ for $j=1,2\cdots,m_2$. Then for each $m_1< i \leq m$,
        \begin{eqnarray}\label{eq: wedge bar u(x) big}
            &&\norm{(e_i\wedge e_{i+1}\wedge\cdots\wedge e_m)\cdot \bar{u}(x)}_p \\
            &=& \norm{(e_i+x'_{i-m_1})\wedge (e_{i+1}+x'_{i+1-m_1})\wedge\cdots\wedge (e_m+x_{m_2}')}_p.\nonumber
        \end{eqnarray}
        In particular, when $i=m_1+1$, we combine \eqref{eq: relation between two norms},  \eqref{eq: bar u(x) wedge 1} and \eqref{eq: wedge bar u(x) big}, and obtain that
        \begin{equation}\label{eq: estimate of det a2(x) big}
            \det(a_2(x)) = t_{m_1+1}(x)t_{m_1+2}(x)\cdots t_m(x) \leq \prod_{j=1}^{m_2} (1+\norm{x_j}^2)^{\frac{1}{2}}
        \end{equation}
        We also have
        \begin{eqnarray}\label{eq: wedge bar u(x) small}
           &&\norm{(e_i+x'_{i-m_1})\wedge (e_{i+1}+x'_{i+1-m_1})\wedge\cdots\wedge (e_m+x_{m_2}')}_p\\
           &\geq& (1+\norm{x_{i-m_1}}^2+\norm{x_{i+1-m_1}}^2+\cdots+\norm{x_{m_2}}^2)^{\frac{1}{2}}.\nonumber
        \end{eqnarray}
        Thus, when $i=m_1+1$, we find
        \begin{eqnarray}\label{eq: estimate of det a2(x) small}
            \det(a_2(x))& =& t_{m_1+1}(x)t_{m_1+2}(x)\cdots t_m(x)\\
            & \geq&  (1+\norm{x_{1}}^2+\norm{x_{2}}^2+\cdots+\norm{x_{m_2}}^2)^{\frac{1}{2}}.\nonumber
        \end{eqnarray}
        This completes the proof of the Lemma.
    \end{proof}
    \begin{proof}[Proof of Theorem \ref{thm: parabolic induction preserves shalika model}]
         To simplify our notation in this proof, we set that $n=n_1+n_2$ and $m=m_1+m_2$. (These $n$ and $m$ have different meanings from those in the Introduction.) Let $\lambda_1$ and $\lambda_2$ be the Shalika functionals of $\pi_1$ and $\pi_2$ respectively. For each positive integer $k$, we write $I_k$ for the $k\times k$ identity matrix. We consider the Weyl element
        \begin{equation*}
            w = \left(
                  \begin{array}{cccc}
                    \RI_{m_1} & 0 & 0 & 0 \\
                    0 & 0 & \RI_{m_2} & 0 \\
                    0 & \RI_{m_1} & 0 & 0 \\
                    0 & 0 & 0 & \RI_{m_2} \\
                  \end{array}
                \right)
        \end{equation*}
        Take any function $f\in V_\pi$, we consider the function $\Phi(g;f)$ on $\GL_{n}$ defined by
        \begin{equation*}
            \Phi(g;f) := \langle \lambda_1\otimes \lambda_2, f(w^{-1}g) \rangle.
        \end{equation*}
        Then $\Phi(g;f) = \Phi(\RI_{n}; \pi(g)f)$. Arguing as \cite[Lemma 3.1]{Mat}, we can easily show that
        \begin{equation}\label{Eq: parabolic induction shalika 1}
            \Phi(\left(
                  \begin{array}{cccc}
                    \RI_{m_1} & 0 & x & z \\
                    0 & \RI_{m_2} & 0 & y \\
                    0 & 0 & \RI_{m_1} & 0 \\
                    0 & 0 & 0 & \RI_{m_2} \\
                  \end{array}
                \right)
            g;f) = \psi(\tr(x)+\tr(y))\Phi(g;f),
        \end{equation}
        and
        \begin{equation}\label{Eq: parabolic induction shalika 2}
            \Phi(\left(
                  \begin{array}{cccc}
                    \RI_{m_1} & x & 0 & 0 \\
                    0 & \RI_{m_2} & 0 & 0 \\
                    0 & 0 & \RI_{m_1} & y \\
                    0 & 0 & 0 & \RI_{m_2} \\
                  \end{array}
                \right)
            g;f) = \Phi(g;f).
        \end{equation}

        Now we consider the integral
        \begin{equation}\label{eq: def of H(g)}
            H(g;f) := \int_{x\in M_{m_2\times m_1}(F)} \Phi(\left(
                  \begin{array}{cccc}
                    \RI_{m_1} & 0 & 0 & 0 \\
                    0 & \RI_{m_2} & x & 0 \\
                    0 & 0 & \RI_{m_1} & 0 \\
                    0 & 0 & 0 & \RI_{m_2} \\
                  \end{array}
                \right)
            g;f)dx.
        \end{equation}
        Then $H(g; f) = H(\RI_n; \pi(g)f).$ As \cite[Lemma 3.2]{Mat}, we aim to show the integral \eqref{eq: def of H(g)} converges absolutely. We also need to prove that the linear map $f\mapsto H(\RI_{n}; f)$ is continuous under the Fr\'{e}chet topology of $V_\pi$. We write the integral \eqref{eq: def of H(g)} as
        \begin{equation*}
            \int_{x\in M_{m_2\times m_1}(F)} \langle \lambda_1\otimes \lambda_2, f(w^{-1}\left(
                  \begin{array}{cccc}
                    \RI_{m_1} & 0 & 0 & 0 \\
                    0 & \RI_{m_2} & x & 0 \\
                    0 & 0 & \RI_{m_1} & 0 \\
                    0 & 0 & 0 & \RI_{m_2} \\
                  \end{array}
                \right)
            g)\rangle dx.
        \end{equation*}
        Replacing $f$ by $\pi(w^{-1}g)f$, we only need to justify the absolute convergence and the continuity property of
        \begin{equation}\label{eq: convergent integral in parabolic induction shalika}
            \int_{x\in M_{m_2\times m_1}(F)} \langle \lambda_1\otimes \lambda_2, f(\left(
                  \begin{array}{cccc}
                    \RI_{m_1} & 0 & 0 & 0 \\
                    0 & \RI_{m_1} & 0 & 0 \\
                    0 & x & \RI_{m_2} & 0 \\
                    0 & 0 & 0 & \RI_{m_2} \\
                  \end{array}
                \right)
            )\rangle dx.
        \end{equation}
        We define $\bar{u}(x)$ to be the $m\times m$-matrix
        \begin{equation}\label{eq: bar u(x)}
            \bar{u}(x) = \mtrtwo{\RI_{m_1}}{}{x}{\RI_{m_2}}.
        \end{equation}
        Let $\bar{u}(x) = u(x)t(x)k(x)$ be the Iwasawa decomposition of $\bar{u}(x)$ in $\GL_m(F)$ as in Lemma \ref{lemma: Iwasawa decomposition lower}. We also define $t_i(x)$, $a_1(x)$, $a_2(x)$ as in Lemma \ref{lemma: Iwasawa decomposition lower}. We further decompose $u(x)$ as
        \begin{equation}
            u(x) = u_3(x)\mtrtwo{u_1(x)}{}{}{u_2(x)},
        \end{equation}
        where $u_1(x)$ and $u_2(x)$ lives in the standard maximal unipotent subgroup of $\GL_{m_1}(F)$ and $\GL_{m_2}(F)$ respectively, and $u_3(x)$ lives in the standard unipotent radical of the parabolic subgroup of $\GL_m(F)$ of type $(m_1, m_2)$. It follows (using the notation introduced in the proof of
        Corollary \ref{lemma: Shalika functional Schwartz}) that
        \begin{equation}\label{eq: 09}
            \begin{aligned}
            &\langle \lambda_1\otimes \lambda_2, f(\iota_c(
                  \begin{pmatrix}
                    \RI_{m_1} & 0 \\
                    x & \RI_{m_2} \end{pmatrix}))\rangle\\
            =&\langle \lambda_1\otimes \lambda_2, f(\iota_c(\begin{pmatrix}
            u_1(x)a_1(x) & \\
            0 & u_2(x)a_2(x) \end{pmatrix}k(x))\rangle\\
            =&\langle \lambda_1\otimes \lambda_2, \delta(a_1(x),a_2(x))
            \pi_1(\iota_2(u_1(x)a_1(x))\pi_2(\iota_1(u_2(x)a_2(x))
            f(\iota_c(k(x))\rangle,\\
        \end{aligned}
        \end{equation}
        where $\delta(a_1(x),a_2(x))$ is the modular character arising from equivariance for the normalized parabolic induction. More precisely,
        \begin{equation}\label{eq: 011}
             \delta(a_1(x),a_2(x)) = (\det a_1(x))^{m_2}\cdot(\det a_2(x))^{-m_1}.
        \end{equation}
        By Corollary \ref{lemma: Shalika functional Schwartz}, there exists a positive integer $M_0$ and a real number $m_0$ with the property that for any positive integer $M\geq M_0$, there exists a continuous seminorm $q_M$ on $V_\pi$, the following estimate holds:
        \begin{equation}\label{eq: 018}
           \begin{aligned}
           &\abs{\langle \lambda_1\otimes \lambda_2, \delta(a_1(x),a_2(x))\pi_1(\iota_2(u_1(x)a_1(x))\pi_2(\iota_1(u_2(x)a_2(x)))
           f(\iota_c(k(x))\rangle}\\
           \leq &\abs{\det (a_1(x))}^{m_2+M-m_0}\cdot\abs{\det(a_2(x))}^{-m_1-M}\cdot q_M(f).
           \end{aligned}
        \end{equation}
        Note that $\det(a_1(x)) \cdot \det(a_2(x)) = 1$. Thus, combining \eqref{eq: 09} and \eqref{eq: 018}, and by the estimate \eqref{eq: estimate of det a2(x)} in Lemma \ref{lemma: Iwasawa decomposition lower}, we can choose a sufficiently large positive integer $M$ such that
        \begin{equation}\label{eq: temp 12}
           \begin{aligned}
           &\langle \lambda_1\otimes \lambda_2, f(\iota_c(
                  \begin{pmatrix}
                    \RI_{m_1} & 0 \\
                    x & \RI_{m_2} \end{pmatrix}))\rangle\\
           \leq &\abs{\det(a_2(x))}^{-m_1-m_2-2M+m_0}\cdot q_M(f)\\
           \leq & (1+\norm{x_{1}}^2+\norm{x_{2}}^2+\cdots+\norm{x_{m_2}}^2)^{\frac{1}{2}(-m_1-m_2-2M+m_0)} \cdot q_M(f)
           \end{aligned}
        \end{equation}
        This implies that the integral \eqref{eq: convergent integral in parabolic induction shalika} converges absolutely and defines a continuous linear functional on $V_\pi$. Thus, the integral \eqref{eq: def of H(g)} converges absolutely and the linear functional $f\mapsto H(\RI_{n}; f)$ is continuous.

        According to \eqref{Eq: parabolic induction shalika 1}, we can see that
        \begin{equation}\label{eq: H(g) psi equivariant}
            H(\RI_{n}; \pi(\mtrtwo{\RI_{m}}{x}{}{\RI_{m}})f) = \psi(\tr(x))H(\RI_n;f).
        \end{equation}
        Denote by $V^-$ the lower unipotent subgroup consisting all matrices of the form
        \begin{equation*}
            v(x):=\left(
                  \begin{array}{cccc}
                    \RI_{m_1} & 0 & 0 & 0 \\
                    0 & \RI_{m_1} & 0 & 0 \\
                    0 & x & \RI_{m_2} & 0 \\
                    0 & 0 & 0 & \RI_{m_2} \\
                  \end{array}
                \right)
        \end{equation*}
        Let $P:=P_{m_1,m_2}$ be the standard parabolic subgroup of $\GL_m(F)$ associated with the partition $m=m_1+m_2$. The group $\GL_m(F)$ diagonally embeds into $\GL_n(F)$ with image $\GL_m^\triangle(F)$. For any subgroup $L$ of $\GL_m(F)$ and any matrix $x\in L$, we also write $L^\triangle$ and $x^\triangle$ for the image of the embedding
        \begin{equation*}
            L \hookrightarrow \GL_m(F) \hookrightarrow \GL_n(F).
        \end{equation*}
        Let $p = \mtrtwo{g_1}{x}{}{g_2}\in P$. Then by \eqref{Eq: parabolic induction shalika 2}, it is easy to see that $H(g;f)$ also satisfies the following equivariant property:
        \begin{equation}\label{eq: 019}
            H(p^\triangle g;f) = \omega(\det p)\delta_P(p) H(g;f)
        \end{equation}
        Thus, we can find a test function $\varphi(h^\triangle;f)\in C^\infty_c(\GL_m^\triangle(F))$ such that
        \begin{equation}
            \begin{aligned}
            H(h^\triangle;f) &= \int_{P} \varphi(p^\triangle h^\triangle;f)\omega^{-1}(\det p) d_lp
            \end{aligned}
        \end{equation}
        where $d_lp$ is the left Haar measure on $P$.
        Now we define a functional $S(f)$ on $V_\pi$ as follows:
        \begin{equation}\label{eq: def of S(g;f)}
            S(f) := \int_{\GL_m(F)} \varphi(h^\triangle ;f)\omega^{-1}(\det h) dh.
        \end{equation}
        We rewrite \eqref{eq: def of S(g;f)} as
        \begin{equation}\label{eq: good formula S(g;f)}
            \begin{aligned}
            S(f) &= \int_{K_m}\int_P \varphi(p^\triangle k^\triangle ;f)\omega^{-1}(\det p)\omega^{-1}(\det k) d_lp dk\\
            &= \int_{K_m} H(k^\triangle ;f)\omega^{-1}(\det k) dk.
            \end{aligned}
        \end{equation}
        It is clear that $S(f)$ is well defined, i.e. the integral in \eqref{eq: def of S(g;f)} is convergent and independent on the choice of $\varphi(h^\triangle ;f)$. Thus combining \eqref{eq: def of S(g;f)}, \eqref{eq: good formula S(g;f)} and \eqref{eq: H(g) psi equivariant}, we see
        \begin{equation*}
            S(\pi(\mtrtwo{h}{}{}{h}\mtrtwo{\RI_m}{x}{}{\RI_m})f) = \omega(\det h) \psi(\tr(x))S(f).
        \end{equation*}
        Thus, the linear map $f\mapsto S(f)$ satisfies the desired equivariant property of a Shalika functional. Next, we show that this linear map is non-zero.

        To show this non-vanishing property, we use the Bruhat decomposition of $\GL_m(F)$ in the integral \eqref{eq: def of S(g;f)}. Let $N^-$ be the unipotent radical of the lower parabolic subgroup associated to the partition $m=m_1+m_2$. Then we have that for all $f\in V_\pi$,
        \begin{equation}\label{eq: temp 01}
            \begin{aligned}
            S(f) &= \int_{N^-}\int_P \varphi(p^\triangle n^\triangle;f)\omega^{-1}(\det p) d_lp dn\\
                   &= \int_{N^-}H(n^\triangle;f) dn\\
                   &= \int_{V^-}\int_{N^-} \langle \lambda_1\otimes \lambda_2, f(v(x)\cdot w^{-1}n^\triangle w \cdot w^{-1}) \rangle dn dv(x)
            \end{aligned}
        \end{equation}
        Note that $V^-\cdot w^{-1}N^-w$ is a subgroup of the unipotent radical $U^-$ of the lower parabolic subgroup of $\GL_n(F)$ associated to the partition $n=n_1+n_2$. We choose a positive test function $\varphi(g)\in C^\infty_c(\GL_n(F))$ whose restriction to the subgroup $V^-\cdot w^{-1}N^-w$ is positive on a subset with positive measure. Since $\lambda_1$, $\lambda_2$ are non-zero Shalika functionals, we can find two vectors $v_1, v_2$ in $V_{\pi_1}$ and $V_{\pi_2}$ resp, such that
        \begin{equation*}
            \langle \lambda_1, v_1 \rangle \ne 0, \qquad \langle \lambda_2, v_2 \rangle \ne 0.
        \end{equation*}
        We define a function $\tilde{f}$ on $U^-$ by
        $\tilde{f}(u^-) = \varphi(u^-)v_1\otimes v_2$,
        and extend it to a function on $\GL_n(F)$ by the equivariant property of the induced representation $\pi$. Then the function
        \begin{equation*}
            f_0(g) := \tilde{f}(gw)
        \end{equation*}
        satisfies
        \begin{equation}
            \begin{aligned}
            S(f_0) &=\int_{V^-}\int_{N^-} \langle \lambda_1\otimes \lambda_2, f_0(v(x)\cdot w^{-1}n^\triangle w \cdot w^{-1}) \rangle\, dn dv(x)\\
                        &=\int_{V^-}\int_{N^-} \langle \lambda_1\otimes \lambda_2, \tilde{f}(v(x)\cdot w^{-1}n^\triangle w) \rangle\,  dn dv(x)\\
                        &=\langle \lambda_1, v_1 \rangle\cdot\langle \lambda_2, v_2 \rangle\cdot\int_{V^-}\int_{N^-} \varphi(v(x)\cdot w^{-1}n{w}) \rangle\, dndv(x)\\
                        &\ne 0
            \end{aligned}
        \end{equation}
        Thus the map $f\mapsto S(f)$ is nonzero.

        We finally check that the linear map $f\mapsto S(f)$ is continuous under the Fr\'{e}chet topology of the induced representation. By \eqref{eq: good formula S(g;f)},
        \begin{equation}
            \begin{aligned}
            S(f) = \int_{K_m} H(\RI_{n} ; \pi(k^\triangle)f)\omega^{-1}(\det k) dk.
            \end{aligned}
        \end{equation}
        For any fixed $k\in K_m$, the linear functional $f\mapsto H(\RI_{n} ; \pi(k^\triangle)f)\omega^{-1}(\det k)$ is continuous.
It follows that with the fixed $f\in V_\pi$, the function $k\mapsto H(\RI_{n} ; \pi(k^\triangle)f)\omega^{-1}(\det k)$ defined on $K_m$ is bounded, since $K_m$ is compact.
Thus, by the Uniform Boundedness Principle, the family of continuous linear functionals $f\mapsto H(\RI_{n} ; \pi(k^\triangle)f)\omega^{-1}(\det k)$ indexed by $k\in K_m$ is equicontinuous. Thus,
        $$S(f) = \int_{K_m} H(\RI_{n} ; \pi(k^\triangle)f)\omega^{-1}(\det k) dk$$
        is a continuous linear functional on $V_\pi$.

    \end{proof}
    \begin{cor}\label{thm-SM}
    Any irreducible essentially tempered Casselman-Wallach representation $(\pi, V_\pi)$ of $G = \GL_{2n}(\R)$ with
    $$\RH^*(\mathfrak{g},K^0Z^0,\pi\otimes F_\nu^\vee)\ne 0$$
    has a non-zero Shalika model defined by the Shalika functional as in \eqref{Eq: Def of Shalika Functional}.
    \end{cor}
    It is clear that Corollary \ref{thm-SM} follows from Theorem \ref{thm: parabolic induction preserves shalika model} and the example on $\GL_2(F)$ discussed at the beginning of this section.


\section{Cohomological Representations and Linear Models}\label{Section: Linear Models}


According to Corollary \ref{thm-SM}, the cohomological representation $\pi$ of $G = \GL_{2n}(\BR)$ as in Theorem \ref{thm-main} has a non-zero Shalika model, and hence the archimedean local integral $Z(v,s,\chi)$ of Friedberg-Jacquet as defined in \eqref{Eq: Def of Local Integral} is nonzero. However, due to a lack of nice formula for the Shalika functional, it is not easy to directly construct a cohomological test vector for the Friedberg-Jacquet integral $Z(v,s,\chi)$. Instead, we construct in this Section an explicit linear model $\Lambda_{s,\chi}$, with which it is easier to obtain a cohomological test vector, as
discussed in the next Section.

\subsection{The $\GL_2(\R)$ case}\label{Subsection:The GL(2,R) Case}
Let us first review the work of A. Popa on the $\GL_2(\R)$ case. We will only recall his results on discrete series, which will be used in our further computation. For detailed discussion on the principal series of $\GL_2(\R)$ and of $\GL_2(\C)$, we refer to his paper \cite{P}.
His work on $\GL_2(\C)$ is also very helpful in the complex case that will be considered in \cite{LT}.

In this Section, we fix a non-trivial unitary character $\psi$ of $\R$ and a character $\chi$ of $\R^\times$. Let $k$ be a positive integer and $D_k$ be the relative discrete series (which is assumed to be a Casselman-Wallach representation) in the Introduction. The minimal $K$-type of $D_k$ (whose central character is $\eta_k$) is denoted by $\tau_k$. Then $\tau_{k}$ is a two dimensional irreducible representation of $\RO_2(\R)$. We can take basis $v_{k}$ and $v_{-k}$ of $\tau_{k}$ such that $v_{k}$ and $v_{-k}$ are weight vectors with weight $\pm(k+1)$ under the action of $\SO_2(\R)$. Moreover, $v_{k}$ and $v_{-k}$ are related by the action of $\epsilon = \mtrtwo{-1}{0}{0}{1}$, i.e.
    $D_{k}(\epsilon)v_{k} = v_{-k}$.
The two dimensional space $\tau_{k}$ contains a one-dimensional subspace denoted by $\tau_{k}^{\epsilon,\chi}$ (in \cite{P}, the notation $\mathcal{W}^{T}$
was used for this invariant subspace), which is spanned by the vector $v_{k}+\chi(-1)v_{-k}$. Note that
\begin{equation*}
     D_k(\epsilon)(v_{k}+\chi(-1)v_{-k}) = \chi(-1)(v_{k}+\chi(-1)v_{-k}).
\end{equation*}

Given an irreducible generic Casselman-Wallach representation $\sigma$ of $\GL_2(\R)$ with a Whittaker model $\mathcal{W}(\sigma,\psi)$, it is clear (\cite{JL70}) that
the archimedean Hecke integral
\begin{equation}\label{eq: local int GL(2,R)}
       \lambda_{s,\chi,\sigma}(v) := \int_{\R^\times} W_v(\mtrtwo{a}{}{}{1})\abs{a}^{s-\frac{1}{2}}\chi(a) d^\times a.
   \end{equation}
has a meromorphic continuation to the whole complex plane. It is a holomorphic multiple of the $L$-function $L(s,\sigma\otimes \chi)$. Whenever $s = s_0$ is not a pole of the $L$-function $L(s, \sigma\otimes \chi)$, $\lambda_{s_0,\chi,\sigma}$ defines a continuous linear functional on $\sigma$. Now if apply to the case of $\sigma = D_k$, then whenever $s = s_0$ is not a pole,
$\lambda_{s_0,\chi, {D_{k}}}$ defines a nonzero element in
\begin{equation*}
     \text{Hom}_{\GL_1(\R)\times \GL_1(\R)}(D_{k}, \abs{\quad}^{\frac{1}{2}-s_0}\chi^{-1}\otimes\abs{\quad}^{s_0-\frac{1}{2}}\eta_k\chi).
\end{equation*}
The non-vanishing property of $\lambda_{s_0,\chi, {D_{k}}}$ can be deduced from the following result of Popa (we will rephrase it using our notation):
\begin{prop}\cite[Theorem 1]{P}
   There exists a vector $v\in \tau_{k}^{\epsilon,\chi}$ such that when $\re(s)$ is sufficiently large,
   \begin{equation}
       \lambda_{s,\chi,D_k}(v) = L(s,D_k\otimes\chi).
   \end{equation}
\end{prop}
Since $\tau_{k}^{\epsilon,\chi}$ is one-dimensional, we can normalize the basis so that
\begin{equation*}
     \lambda_{s,\chi,{D_{k}}}(v_{k}+\chi(-1)v_{-k}) = L(s, D_{k}\otimes\chi) .
\end{equation*}
The same result also holds once we twist a determinant character on the relative discrete series.

\begin{cor}\label{Cor: GL(2, R) Non-vanishing}
    Given $\alpha\in\C$, let $\sigma_k := D_k\otimes \abs{\cdot}^{\alpha}$ and $\lambda_{s,\chi,{\sigma_{k}}}$ be the continuous linear functional defined in \eqref{eq: local int GL(2,R)}. Then for every $v\in V_{\sigma_k}$, $\lambda_{s,\chi,{\sigma_{k}}}(v)$ is a holomorphic multiple of $L(s,\sigma_k\otimes \chi)$. Whenever $s = s_0$ is not a pole of $L(s, \sigma_{k}\otimes \chi)$, $\lambda_{s_0,\chi,{\sigma_{k}}}(v)$ defines a nonzero element in \begin{equation*}
     \Hom_{\GL_1(\R)\times \GL_1(\R)}(\sigma_{k}, \abs{\cdot}^{\frac{1}{2}-s_0}\chi^{-1}\otimes\abs{\cdot}^{s_0-\frac{1}{2}}\tilde{\eta}_k\chi),
    \end{equation*}
    where $\tilde{\eta_k} = \eta_k\abs{\cdot}^{\alpha}$ is the central character of $\sigma_k$.
    Moreover,
    \begin{equation}\label{Eq: Precise L(0.5, D(k+1))}
       \lambda_{s,\chi,{\sigma_{k}}}(v_{k}+\chi(-1)v_{-k}) = L(s,\sigma_k\otimes \chi).
\end{equation}
\end{cor}

\subsection{A new construction of linear model}\label{Subsection: Another Linear Model}
In this Subsection, we retain the notation in the Introduction and assume that $n\geq 2$. Let $\pi$ be the parabolically induced representation in \eqref{Eq: pi parabolic induction parameter}. In Section \ref{section: existence of Shalika model}, we have shown that $\pi$ has a Shalika model (Corollary \ref{thm-SM}). Thus the local integral \eqref{Eq: Def of Local Integral} defines a twisted linear model of $\pi$, i.e. \eqref{Eq: Def of Local Integral} defines a non-zero element in
\begin{equation*}\text{Hom}_{H}(\pi, \abs{\det}^{-s+\frac{1}{2}}\chi^{-1}(\det)\otimes\abs{\det}^{s-\frac{1}{2}}(\chi\omega)(\det))
\end{equation*}
whenever $s$ is not a pole of the $L$-function, as we explained in the Introduction. The goal of this Subsection is to construct another linear model without using the Shalika model.

The basic idea comes from standard Bruhat theory. We first define a Weyl element $w$. Let $e_1, e_2,\cdots, e_{2n}$ be the standard basis of the vector space $\BR^{2n}$, where $\BR^{2n}$ is realized as a vector space of column vectors. We consider the following Weyl element
    \begin{equation}\label{eq: def of w}
        w := (e_1, e_3, \cdots, e_{2n-1}, e_2, e_4, \cdots, e_{2n}).
    \end{equation}
Denote by $P$ the standard parabolic subgroup of $G$ corresponding to the partition $[2^n]$ of $2n$, as before.
We look at the homogeneous space $P\quo G$ and let $H$ acts on $P\quo G$ by right translation. It was shown in \cite[Proposition 3.4]{Mat1} that there are finitely many orbits. In this paper, it is sufficient for us to consider the orbit $P\quo PwH$. By a direct matrix computation, it is easy to show that the stabilizer is
$$w^{-1}Pw \cap H = B_1\times B_2,$$
where $B_1\times B_2$ is the standard upper Borel subgroup of $H = \GL_n(\BR)\times\GL_n(\BR).$ Thus the orbit $P\quo PwH $ is homeomorphic to $(B_1\times B_2)\quo H$. In particular, the orbit $P\quo PwH$ is closed. In the following, we are going to construct an integral on this closed orbit which represents a linear functional in
\begin{equation*}\text{Hom}_{H}(\pi, \abs{\det}^{-s+\frac{1}{2}}\chi^{-1}(\det)\otimes\abs{\det}^{s-\frac{1}{2}}(\chi\omega)(\det)).
\end{equation*}

We still fix a character $\chi$ of $\R^\times$ as we did in Subsection \ref{Subsection:The GL(2,R) Case}.
 Let $P=MU$ be the Levi decomposition of $P$.  Define $H' = \Ad(w)H$. Then $H'$ is also isomorphic to $\GL_n(\R)\times \GL_n(\R)$. Let
    $N = N_1\times N_2$ be the unipotent radical of $B_1\times B_2$. The standard maximal unipotent subgroup $N$ of $H$ is mapped onto the standard maximal unipotent subgroup $N'$ of $H'$. We notice that $N'$ is a subgroup of $U$.

Take $\varphi\in V_\pi$. Then $\varphi$ is a smooth function on $G$ with value in $\widehat{\bigotimes}_{i=1}^n V_{\sigma_i}$, where $\sigma_i:= D_{l_i}\abs{\quad}^{\frac{m}{2}}$. As we explained in the Introduction, all $l_i$ have the same parity, and hence all $\sigma_i$ have the same central character $\omega$ (defined in \eqref{eq: omega char}). For each $i$, whenever $s=s_0$ is not a pole of $L(s,\sigma_i\otimes\chi)$, we denote by $\lambda_{s_0,i}$ the non-zero element in
\begin{equation*} \text{Hom}_{\GL_1(\R)\times \GL_1(\R)}(\sigma_{i}, \abs{\quad}^{\frac{1}{2}-s_0}\chi^{-1}\otimes\abs{\quad}^{s_0-\frac{1}{2}}\omega\chi)
\end{equation*}
as in Corollary \ref{Cor: GL(2, R) Non-vanishing}.

To simplify notation, we define two characters of $\BR^\times$:
\begin{equation}\label{def: chi1 and chi2}
    \begin{aligned}
    \chi_{1,s}(a) &:= \abs{a}^{\frac{1}{2}-s}\chi^{-1}(a),\\
    \chi_{2,s}(a) &:= \abs{a}^{s-\frac{1}{2}}\omega(a)\chi(a).
    \end{aligned}
\end{equation}
Let us  consider a function on $G$ defined by
\begin{equation}\label{Eq: def of F on H}
    F_s(g; \varphi) = \langle \bigotimes_{i=1}^n \lambda_{s,i}, \varphi(wg)\rangle.
\end{equation}
Then $F_s(g; \varphi) = F_s(\RI_{2n}; \pi(g)\varphi)$. Since $N'$ is a subgroup of $U$, it is easy to check that for $n_1\in N_1$, $n_2\in N_2$,
\begin{equation*}
    \begin{aligned}
    F_s(\mtrtwo{n_1}{}{}{n_2}g; \varphi)= F_s(g; \varphi).
    \end{aligned}
\end{equation*}
By the equivariance of $\lambda_{s,i}$, it is also easy to check the equivariance of $F(g; \varphi)$ on the torus: for all $n\times n$ diagonal invertible matrices $a_1, a_2$, we have that
\begin{equation}\label{Eq: 03}
    \begin{aligned}
    F_s(\mtrtwo{a_1}{}{}{a_2}g; \varphi)
    = \delta_P^{\frac{1}{2}}(w\mtrtwo{a_1}{}{}{a_2}w^{-1})\cdot\chi_{1,s}(\det a_1)\cdot\chi_{2,s}(\det a_2)F_s(g; \varphi),
    \end{aligned}
\end{equation}
where the modular character can be explicitly computed as follows
\begin{equation}
    \delta_P^{\frac{1}{2}}(w\mtrtwo{a_1}{}{}{a_2}w^{-1}) = \delta_{B_1}(a_1)\delta_{B_2}(a_2).
\end{equation}
It follows that $F_s(g)$ satisfies a $B_1\times B_2-$equivariant property: for any $(b_1, b_2)\in B_1\times B_2,$
\begin{equation}\label{Eq: equivariant property of F}
     F_s(\mtrtwo{b_1}{}{}{b_2}g; \varphi) = \delta_{B_1}(b_1)\delta_{B_2}(b_2)\chi_{1,s}(\det b_1)\cdot\chi_{2,s}(\det b_2)\cdot F_s(g; \varphi).
\end{equation}
Thus the following convergent integral defines a continuous linear functional $\Lambda_{s,\chi}$ on $V_\pi$
$$
\begin{aligned}
        \Lambda_{s,\chi}(\varphi):=
         & \int_{K\cap H} F_s(\mtrtwo{k_1}{}{}{k_2}; \varphi)\chi_{1,s}^{-1}(\det k_1)\chi_{2,s}^{-1}(\det k_2)) dk_1dk_2.
   \end{aligned}
$$
It is easy to see that $\Lambda_{s,\chi}\in\textrm{Hom}_H(\pi, \chi_{1,s}(\det)\otimes\chi_{2,s}(\det))$.

In terms of $\varphi$, $\Lambda_{s,\chi}(\varphi)$ can be written as:
\begin{eqnarray}\label{eq: new def H inv linear functional}
   \Lambda_{s,\chi}(\varphi) &=& \int_{H\cap K} \langle \bigotimes_{i=1}^n \lambda_{s,i}, \varphi(w\mtrtwo{k_1}{}{}{k_2}) \chi_{1,s}^{-1}(\det k_1)\chi_{2,s}^{-1}(\det k_2) \rangle dk_1dk_2\nonumber\\
                    &=& \langle \bigotimes_{i=1}^n \lambda_{s,i}, \tilde{\varphi}(w) \rangle,
\end{eqnarray}
where $\tilde{\varphi}$ is obtained by averaging $\varphi$ against $\chi_{1,s}^{-1}(\det k_1)\chi_{2,s}^{-1}(\det k_2)$ over
the compact group $K\cap H$. In particular, if $\varphi$ satisfies the right $K\cap H$-equivariant property:
\begin{equation}\label{eq: right equivariant property}
    \varphi(g\mtrtwo{k_1}{}{}{k_2}) = \chi_0(\det k_1)\cdot(\chi_0\omega_0)(\det k_2)\cdot\varphi(g),
\end{equation}
where $\chi_0$ ($\omega_0$ resp.) is the restriction of the character $\chi$ ($\omega$ resp.) on $\{\pm 1\}$, then
\begin{equation}\label{eq: 020}
    \Lambda_{s,\chi}(\varphi) = \langle \bigotimes_{i=1}^n \lambda_{s,i}, \tilde{\varphi}(w) \rangle =\langle \bigotimes_{i=1}^n \lambda_{s,i}, \varphi(w) \rangle
\end{equation}
The following Proposition gives the desired property of $\Lambda_{s,\chi}$.
\begin{prop}\label{Cor: analytic property of Lambda}
    For every $\varphi\in V_\pi$, $\Lambda_{s,\chi}(\varphi)$ defined by \eqref{eq: new def H inv linear functional} has a meromorphic continuation in $s$ to the whole complex plane. It is a holomorphic multiple of $L(s,\pi\otimes\chi)$ and defines an element in the space
    \begin{equation*}
              \Hom_H(\pi, \chi_{1,s}(\det)\otimes\chi_{2,s}(\det))
    \end{equation*}
    which is the same as $\Hom_{H}(\pi, \abs{\det}^{-s+\frac{1}{2}}\chi^{-1}(\det)\otimes\abs{\det}^{s-\frac{1}{2}}(\chi\omega)(\det))$,
    whenever $s$ is not a pole of the $L$-function.
\end{prop}
\begin{proof}
   This is a direct consequence of Corollary \ref{Cor: GL(2, R) Non-vanishing}.
\end{proof}
We close this Section by the following remark:
\begin{rk}
   One should be able to prove that $\Lambda_{s,\chi}$ is a non-zero linear functional by imitating the proof of Theorem \ref{thm: parabolic induction preserves shalika model}. Since later, we will prove a sharper result that $\Lambda_{s,\chi}$ does not vanish on the minimal $K$-type $\tau$ of $\pi$, we will not discuss the non-vanishing property here.
\end{rk}


\section{Cohomological Vectors in the Induced Representation}\label{Section: Cohomological vector realization}


The goal of this section is to explicitly construct a cohomological vector of $\pi$ with the desired non-vanishing property for Theorem \ref{thm-main}. We retain all notation in the Introduction. Since $\pi$ (given in \eqref{Eq: pi parabolic induction parameter}) is realized as a parabolically induced representation, every function $f\in V_\pi$ is determined by its value on the maximal compact subgroup $K$. As \eqref{eq: 020} suggests, we only care about the value of cohomological function at the Weyl element $w\in K$ as defined in \eqref{eq: def of w}. Thus, it is convenient to work with the compact induction model of $\pi$. Let us start with some reductions and outline our strategy of the construction of the function in the minimal $K$-type.


\subsection{Some reductions}\label{subsection: reduction to polynomial repn}


Let $\pi$ be the parabolically induced representation in \eqref{Eq: pi parabolic induction parameter}, where $P$ is the standard parabolic subgroup of $G$ corresponding to the partition $[2^n]$ with Levi decomposition $P=MU$. Then $M\cap K$ is just a product of $n$ copies of $\RO_2(\R)$.
We write a general element in $(M\cap K)^0$ as
    \begin{equation}\label{Eq: parametrize M intersect K 1}
       k(\theta_1,\theta_2,\cdots,\theta_n) = \text{diag}(\mtrtwo{\cos{\theta_1}}{\sin{\theta_1}}{-\sin{\theta_1}}{\cos{\theta_1}},\cdots,\mtrtwo{\cos{\theta_n}}{\sin{\theta_n}}{-\sin{\theta_n}}{\cos{\theta_n}}).
    \end{equation}
    For the component group of $M\cap K$, we write
    \begin{equation}\label{Eq: parametrize M intersect K 2}
       c(\epsilon_1,\cdots,\epsilon_n) = \text{diag}(\mtrtwo{-1}{0}{0}{1}^{\frac{\epsilon_1-1}{2}},\cdots,\mtrtwo{-1}{0}{0}{1}^{\frac{\epsilon_n-1}{2}}),
    \end{equation}
    where all $\epsilon_i\in \{\pm 1\}$. For simplicity, for any representation $\rho_0$ of $\SO_2(\R)$, we write $\rho_0(\theta)$ for $\rho_0(\mtrtwo{\cos{\theta}}{\sin{\theta}}{-\sin{\theta}}{\cos{\theta}})$.
We denote by $\chi_l$, for each integer $l$, the character of $\SO_{2}(\R)$
    that sends $\mtrtwo{\cos\theta}{\sin\theta}{-\sin\theta}{\cos\theta}$ to $e^{i\cdot l\theta}$.

    Set $\sigma_j := D_{l_j}\abs{\quad}^\frac{m}{2}$, and $\lambda_j :=\lambda_{s,\chi,\sigma_j}$ to be the continuous linear functional on $\sigma_j$ defined
by the local Hecke integral \eqref{eq: local int GL(2,R)}. Let $\tau_j$ be the minimal $K$-type of $\sigma_j$. As described in Subsection \ref{Subsection:The GL(2,R) Case}, $\tau_j$ is a two dimensional space, in which there exists a basis $\{v_j,v_{-j}\}$ with property that
    \begin{enumerate}\label{property of minimal K type basis}
      \item $v_{-j} = \sigma_j(\mtrtwo{-1}{}{}{1})v_j = \tau_j(\mtrtwo{-1}{}{}{1})v_j$;
      \item $\sigma_j(\theta)v_j = \tau_j(\theta)v_j = e^{i\cdot(l_j+1)\theta}v_j$ and $\sigma_j(\theta)v_{-j} = \tau_j(\theta)v_{-j} = e^{-i\cdot(l_j+1)\theta}v_{-j}$;
      \item $\langle \lambda_j, v_j+\chi(-1)v_{-j} \rangle = L(s, \sigma_j\otimes\chi)$.
    \end{enumerate}
Denote by $\tau_j^{\epsilon,\chi}$ the one dimensional subspace of $\tau_j$ spanned by the vector $v_j+\chi(-1)v_{-j}$. By \cite[Proposition 8.1]{V2}, the minimal $K$-type $\tau$ of $\pi$ is also the minimal $K$-type in the induced representation
\begin{equation}\label{eq: pi 0}
    \pi_0 := \text{Ind}_{M\cap K}^{K} \tau_1\otimes\cdots\otimes\tau_n.
\end{equation}
 Every function $f\in V_{\pi_0}$ is a smooth function in $C^\infty(K,\tau_1\otimes\cdots\otimes\tau_n)$ satisfying the equivariant property:
    \begin{equation}\label{eq: equivariant property on functions on matrix}
         f(\diag(k_1,k_2,\cdots,k_n)k) = (\otimes_{j=1}^n\tau_j(k_j)) f(k).
    \end{equation}
    It may not be convenient to work with vector-valued functions. By using the basis of $\tau_1\otimes\tau_2\otimes\cdots\otimes\tau_n$, we may obtain
    scalar valued functions as follows: For every $f\in C^\infty(K,\tau_1\otimes\cdots\otimes\tau_n)$, we can write
    \begin{equation}\label{Eq: basis expansion of phi}
       f(k) = \sum f_{\eta_1,\eta_2,\cdots,\eta_n}(k) v_{\eta_1}\otimes v_{\eta_2}\otimes\cdots\otimes v_{\eta_n},
    \end{equation}
    where the summation is taken over all possible choices $\eta_j\in \{\pm j\}$ and $f_{\eta_1,\eta_2,\cdots,\eta_n}(k)\in C^\infty(K)$.
    \begin{lemma}\label{Lemma: compact induction lemma}
       A smooth function $f\in C^\infty(K,\tau_1\otimes\cdots\otimes\tau_n)$ with basis expansion \eqref{Eq: basis expansion of phi} satisfies the equivariant property \eqref{eq: equivariant property on functions on matrix}  if and only if both
       \begin{equation}\label{Eq: left eigen}
           f_{\eta_1,\eta_2,\cdots,\eta_n}(k(\theta_1,\theta_2,\cdots,\theta_n)k) = (\prod_{j=1}^{n} e^{i\cdot(l_j+1)\cdot \sgn(\eta_j)\theta_j}) f_{\eta_1,\eta_2,\cdots,\eta_n}(k),
       \end{equation}
       and
       \begin{equation}\label{Eq: component inv}
           f_{\eta_1,\eta_2,\cdots,\eta_n}(c(\epsilon_1,\epsilon_2,\cdots,\epsilon_n)k) = f_{\epsilon_1\eta_1,\epsilon_2\eta_2,\cdots,\epsilon_n\eta_n}(k)
       \end{equation}
       hold for all possible choices of $\epsilon_i\in\{\pm 1\}$ and $\eta_j\in \{\pm j\}$. Consequently, the map $f\mapsto f_{1,2,\cdots,n}$ defines a $K$-module isomorphism between $\pi_0$ and $\Ind_{(M\cap K)^0}^K\chi_{l_1+1}\otimes\chi_{l_2+1}\otimes\cdots\otimes\chi_{l_n+1}.$
    \end{lemma}
    \begin{proof}
       Let us assume that $f$ satisfies the equivariant property \eqref{eq: equivariant property on functions on matrix}. Then
       $f(k(\theta_1,\theta_2,\cdots,\theta_n)k)$ is equal to
       \begin{equation*}
          \begin{aligned}
          &(\otimes_{j=1}^{n}\tau_j(\theta_j)) \sum_{\text{all } \eta_j\in\{\pm j\}} f_{\eta_1,\eta_2,\cdots,\eta_n}(k) v_{\eta_1}\otimes v_{\eta_2}\otimes\cdots\otimes v_{\eta_n}\\
          = &\sum_{\text{all }\eta_j\in\{\pm j\}} f_{\eta_1,\eta_2,\cdots,\eta_n}(k)\cdot (\prod_{j=1}^{n} e^{i\cdot(l_j+1)\cdot \text{sgn}(\eta_j)\theta_j})\cdot v_{\eta_1}\otimes v_{\eta_2}\otimes\cdots\otimes v_{\eta_n}.\\
          \end{aligned}
       \end{equation*}
       By comparing the coefficients of each basis vector $v_{\eta_1}\otimes v_{\eta_2}\otimes\cdots\otimes v_{\eta_n}$, we get \eqref{Eq: left eigen}.
The equivariant property of $f$ with respect to $c(\epsilon_1,\cdots,\epsilon_n)$ can be obtained by the following calculation:
       \begin{equation*}
           \begin{aligned}
          &f(c(\epsilon_1,\epsilon_2,\cdots,\epsilon_n)k)\\
          =& (\otimes_{j=1}^{n}\tau_j(\mtrtwo{-1}{0}{0}{1}^{\frac{\epsilon_j-1}{2}})) \sum_{\text{all } \eta_j\in\{\pm j\}} f_{\eta_1,\eta_2,\cdots,\eta_n}(k) v_{\eta_1}\otimes v_{\eta_2}\otimes\cdots\otimes v_{\eta_n}\\
          =&\sum_{\text{all } \eta_j\in\{\pm j\}} f_{\eta_1,\eta_2,\cdots,\eta_n}(k) v_{\epsilon_1\eta_1}\otimes v_{\epsilon_2\eta_2}\otimes\cdots\otimes v_{\epsilon_n\eta_n}\\
          =&\sum_{\text{all } \eta_j\in\{\pm j\}} f_{\epsilon_1\eta_1,\epsilon_2\eta_2,\cdots,\epsilon_n\eta_n}(k) v_{\eta_1}\otimes v_{\eta_2}\otimes\cdots\otimes v_{\eta_n}.
          \end{aligned}
       \end{equation*}
       By comparing the coefficients for each basis vector, we get \eqref{Eq: component inv}. On the other hand, since $\RO_2(\R)$ is a semidirect product of $\SO_2(\R)$ and $\mathbb{Z}/2\mathbb{Z}$, once a function $f(k)$ satisfies \eqref{Eq: left eigen} and \eqref{Eq: component inv},
       one must have that $f(k)$ satisfies the desired equivariant property \eqref{eq: equivariant property on functions on matrix}. This completes our proof.
    \end{proof}

    Under the $K$-isomorphism in Lemma \ref{Lemma: compact induction lemma}, once we construct a complex-valued function $f_{1,2,\cdots,n}$ in the minimal $K$-type $\tau_0$ of  induced space
    \begin{equation}\label{eq: induced space complex value} \pi_{\textrm{sc}} := \Ind_{(M\cap K)^0}^K\chi_{l_1+1}\otimes\chi_{l_2+1}\otimes\cdots\otimes\chi_{l_n+1}\sbst C^\infty(K),\end{equation}
    we can use the equivariant properties \eqref{Eq: left eigen}, \eqref{Eq: component inv} and the basis expansion \eqref{Eq: basis expansion of phi} to recover a vector-valued function $f$ in the minimal $K$-type $\tau$ of $\pi_0$.

   \subsection{On certain minimal $K$-type functions}\label{subsection: Construction of the Scalar-valued Function in the Minimal K-type}
   As explained in the previous subsection, we only need to construct a scalar-valued function in the minimal $K$-type of the induced representation $\pi_{\textrm{sc}}$ in \eqref{eq: induced space complex value}. It turns out that our construction fits in a more general framework that may possibly be useful in other cases. Thus we would like to discuss this general framework in this subsection and provide a detailed formula of the desired cohomological test vector in $\pi_0$ (see \eqref{eq: pi 0}) in the next subsection.

   In this Subsection, we set $K'$ to be either a compact special unitary group $\SU_m$ or a compact unitary group $\RU_m$, a compact special orthogonal group $\SO_{m}(\BR)$ or a compact orthogonal group $\RO_m(\BR)$,
or a compact symplectic group $\Sp(2m)$, for any integer $m\geq 1$.
Or even more generally we may take $K'$ to be a finite product of those compact Lie groups.
We fix once and for all a maximal torus $T'$ of $(K')^0$,
and obtain a positive root system $\Phi^+((\fk')^\BC, (\ft')^\BC)$. It is well-known that all irreducible representations of $(K')^0$ can be parameterized by the highest weights, by the standard highest weight theory of compact groups \cite[Theorem 5.110]{Kn1}. Such a highest weight can be identified with a certain number of integers in the decreasing order.
   \begin{rk}
       To be precise, the highest weights which we used in this paper are analytically integral.
   \end{rk}
   When passing from $(K')^0$ to $K'$, we need to clarify our parametrization of irreducible representations of each disconnected factor of $K'$.
It is enough to clarify the 'highest weight' of a irreducible representation of $\RO_m(\BR)$ $(m=1,2,\cdots)$, as they are the only disconnected simple compact Lie groups among the list of compact groups
we considered in the previous paragraph.

   For odd orthogonal groups $\RO_{2l+1}(\BR)$ ($l=0,1,2 \cdots$), as
   $$ \RO_{2l+1}(\BR) \simeq \SO_{2l+1}(\BR)\times (\BZ/2\BZ),$$
   any irreducible $\RO_{2l+1}(\BR)$-module $\sigma$ is parameterized by $\mu[\xi]$, where
   $\mu = (\mu_1\geq \mu_2 \geq\cdots\geq \mu_l)\in\BZ^l$
   is the highest weight of $\sigma_0$ when restricted on $\SO_{2l+1}(\BR)$; and $\xi = \id$ or $\sgn$ is a quadratic character of the component group $\RO_{2l+1}(\BR)/\SO_{2l+1}(\BR)$. We call $\nu = \mu[\id]$ or $\mu[\sgn]$ to be the highest weight of $\sigma$.
   \begin{rk}
      In particular, when $l=0$, we just say that $[\id]$ or $[\sgn]$ are highest weights of the one-dimensional representation of the group $\RO_1(\BR) = \BZ/2\BZ.$
   \end{rk}

   For even orthogonal groups $\RO_{2l}(\BR)$ ($l=1,2 \cdots$), the restriction of an irreducible $\RO_{2l}(\BR)$-module $\sigma$ to $\SO_{2l}(\BR)$ is either irreducible, or reducible with two irreducible summands. If $\sigma_0:= \sigma\lvert_{\SO_{2l}(\BR)}$ is irreducible, then the highest weight of $\sigma_0$ has the form
   $$ \mu = (\mu_1 \geq \mu_2 \geq \cdots \geq \mu_{l-1} \geq 0)\in\BZ^l.$$
   In this case, there are exactly two non-equivalent irreducible $\RO_{2l}(\BR)$-modules whose restriction to $\SO_{2l}(\BR)$ has the above highest weight.
To distinguish them, we call $\nu:= \mu[\eps]$ with $\eps = 1$ or $-1$ to be the highest weight $\sigma$. If $\sigma_0$ is reducible, then
   it decomposes into two irreducible $\SO_{2l}(\BR)$-modules with highest weights
   $$ (\mu_1, \mu_2, \cdots, \mu_{l-1}, \mu_l), \quad (\mu_1, \mu_2, \cdots, \mu_{l-1}, -\mu_l)\in \BZ^l,$$
   where
   $$ \mu_1 \geq \mu_2 \geq \cdots \geq \mu_{l-1} \geq \mu_l >0. $$
   In this case, we call
   $$\nu := (\mu_1 \geq \mu_2 \geq \cdots \geq \mu_{l-1} \geq \mu_l)\in\BZ^l$$
   to be the highest weight of $\sigma$.
   \begin{rk}
      We would like to point out that the above summary for orthogonal groups is well-known, for example, see \cite[Section 5.5.5]{GW09}.
   \end{rk}
   Accordingly, we have the notion of highest weight vector, and similarly the notion of lowest weight vector of an irreducible representation of orthogonal groups. Thus, the notion of highest and lowest weight vectors of $K'$ are now clear. When $\beta$ is an irreducible $K'$-module with highest weight $\nu$, we define $\chi_\nu$ to be the character of $T'$ on the one-dimensional space generated by any nonzero highest weight vector of $\beta$.

   Next, we will clarify the notion of Cartan component of the tensor product of two irreducible $K'$-modules. It suffices to clarify that for the orthogonal groups $\RO_m(\BR)$, as the Cartan component of the tensor product of two irreducible representations of connected compact groups is well-known. Suppose that $(\alpha, W_\alpha)$ and $(\beta, W_\beta)$ are two irreducible $\RO_{m}(\BR)$-modules with two highest weight vectors $v_\alpha\in W_\alpha, v_\beta\in W_\beta$. Then there exists a unique irreducible $\RO_m(\BR)$-submodule of the tensor product $W_\alpha\otimes W_\beta$, called the Cartan component of $W_\alpha\otimes W_\beta$, generated by $v_\alpha\otimes v_\beta$. It is clear now that one can extend the notion of Cartan component of the tensor product of two irreducible $K'$-modules in a natural way.

   For any irreducible $K'$-module $(\beta, W)$ with highest weight $\nu$, we write $(\beta^*, W^*) $ for its dual representation and  $\nu^*$ for the highest weight of $\beta^*$. It is well-known that as a $K'$-module on the left, the space of $K'$-finite vectors $C^\infty(K')_{\textrm{fin}}$ is completely reducible, and
       \begin{equation}
            C^\infty(K')_{\textrm{fin}} = \bigoplus_{\alpha\in\hat{K'}} \alpha\otimes \Hom_{K'}(\alpha, C^\infty(K')_{\textrm{fin}}),
       \end{equation}
       where $\hat{K'}$ stands for the unitary dual of $K'$ as usual. Under the $K'$-action of $C^\infty(K')_{\fin}$ on the right, the space
       $\Hom_{K'}(\alpha, C^\infty(K')_{\textrm{fin}})$ carries a natural $K'$-action such that
       $$ \Hom_{K'}(\alpha, C^\infty(K')_{\textrm{fin}}) \simeq \alpha^*$$
       as $K'$-modules.
   We summarize the above well-known result of representation of compact groups as the following Lemma, which will be used repeatedly in this Section.
   \begin{lemma}\label{lemma: useful dual lemma}
       Suppose that under the left action of $K'$, a function $f(x)\in C^\infty(K')$ generates an irreducible $K'$-submodule $\alpha$ of $C^\infty(K')$, then under the right $K'$-action, it generates an irreducible $K'$-submodule of $C^\infty(K')$ which is isomorphic to the dual representation $\alpha^*$.
   \end{lemma}

   Suppose further that we have a closed subgroup $C'$ of $K'$.
   \begin{prop}\label{prop: construction of fundamental weight building function}
   Let $\lambda_{W, \textrm{endo}}$ be a $C'$-invariant linear functional on $W^*\otimes W$, i.e.
   $$ \langle\lambda_{W,\, \textrm{endo}}, (\beta^*\otimes\beta)(c') v \rangle = \langle\lambda_{W,\, \textrm{endo}}, v \rangle$$
   for any $c'\in C'$, $v\in W^*\otimes W.$ Then for any highest weight vectors $w\in W$, $w^*\in W^*$, the following function
   \begin{equation}\label{eq: fw}
       f_{W}(x) := \langle \lambda_{W, \textrm{endo}}, \beta^*(x^{-1})w^*\otimes\beta(x^{-1})w\rangle, \quad x\in K',
   \end{equation}
   is right $C'$-invariant and lives in the induced representation $\Ind_{T'}^{K'} \chi_{\nu}^{-1}\chi_{\nu^*}^{-1}$. Moreover, if $f_W$ is nonzero, it generates a minimal $K'$-type $\tau'$ of
the induced representation $\Ind_{T'}^{K'} \chi_{\nu}^{-1}\chi_{\nu^*}^{-1}$, in which $f_W$ is a highest weight vector. In this case, the minimal $K'$-type $\tau'$ is isomorphic to the Cartan component of $W^*\otimes W$.
   \end{prop}
   \begin{proof}
       Right $C'$-invariance of $f_W$ follows directly from the $C'$-invariance of the linear functional $\lambda_{W,\,\textrm{endo}}$. For any $t\in T'$, we have
       \begin{equation}\label{eq: equiv fw}
           \begin{aligned}
                  f_{W}(tx) = &\langle \lambda_{W, \textrm{endo}}, \beta^*(x^{-1}t^{-1})w^*\otimes\beta(x^{-1}t^{-1})w\rangle\\
                            = & \chi_{\nu}^{-1}(t)\chi_{\nu^*}^{-1}(t)\cdot\langle \lambda_{W, \textrm{endo}}, \beta^*(x^{-1})w^*\otimes\beta(x^{-1})w\rangle\\
                            = & \chi_{\nu}^{-1}(t)\chi_{\nu^*}^{-1}(t) f_{W}(x).
           \end{aligned}
       \end{equation}
       Thus $f_{W}(x)\in \Ind_{T'}^{K'} \chi_{\nu}^{-1}\chi_{\nu^*}^{-1}.$ By the Frobenius reciprocity law, the Cartan component of $W^*\otimes W$ is isomorphic to a minimal $K$-type of $\Ind_{T'}^{K'} \chi_{\nu}^{-1}\chi_{\nu^*}^{-1}$. It remains to check the last claim in the Proposition.

       Let $w_1\in W$ and $w_1^*\in W^*$ be any two vectors. We define a smooth function on $K'$ by
       \begin{equation}\label{eq: fw temp}
       f_{W}(x; w_1, w_1^*) := \langle \lambda_{W,\, \textrm{endo}}, \beta^*(x^{-1})w_1^*\otimes\beta(x^{-1})w_1\rangle.
       \end{equation}
       It is straightforward to check that under the left $K'$-action of $C^\infty(K')$, the map
       \begin{equation*}
           \begin{aligned}
                W^*\otimes W &\rightarrow C^\infty(K')\\
                w_1^*\otimes w_1 &\mapsto f_{W}(x; w_1, w_1^*)
           \end{aligned}
       \end{equation*}
       is a $K'$-intertwining operator. Together with \eqref{eq: equiv fw}, this implies that if $f_{W}(x)$ defined in \eqref{eq: fw} is nonzero, then under the left action of $K'$,
it generates an irreducible $K'$-submodule of $C^\infty(K')$ isomorphic to the Cartan component of $W^*\otimes W$. Moreover, $f_{W}(x)$ is a highest weight vector in this irreducible $K'$-submodule.
By Lemma \ref{lemma: useful dual lemma}, under the right $K'$-action, $f_{W}(x)$ generates an irreducible $K'$-submodule of $C^\infty(K')$ isomorphic to dual of the Cartan component of $W^*\otimes W$.
The Proposition finally follows from the simple observation that $W^*\otimes W$ is a self-dual $K'$-module.
   \end{proof}
   Similarly, we have the following Corollary, whose proof is the same as that of Proposition \ref{prop: construction of fundamental weight building function}, and will be omitted here.
   \begin{cor}\label{cor: weight building function in general}
       Given $n$ irreducible $K'$-module $W_1, W_2,$ $\cdots W_n$, for each $j=1,2,\cdots,n$, let $\nu_j$ be the highest weight of $W_j$ and a $C'$-invariant function $f_{W_j}(x)$ be as in \eqref{eq: fw}. For $n$ non-negative integers $m_1, m_2,\cdots, m_n$, define a character of $T'$ via
       $$\chi_\nu:= \prod_{j=1}^n (\chi_{\nu_j}\chi_{\nu_j^*})^{m_j}.$$ Then the smooth function
       $$f(x; W_1, W_2, \cdots, W_n):= \prod_{j=1}^n f_{W_j}^{m_j}(x)$$
       is a $C'$-invariant function in the induced representation $\Ind_{T'}^{K'} \chi_{\nu}^{-1}$. Moreover, if $f(x; W_1, W_2, \cdots, W_n)$ is nonzero, it generates a minimal $K'$-type $\tau'$ of $\Ind_{T'}^{K'} \chi_{\nu}^{-1},$ in which $f(x; W_1, W_2, \cdots, W_n)$ is a highest weight vector. In this case, the minimal $K'$-type $\tau'$ is isomorphic to the Cartan component of
       $$\bigotimes_{j=1}^n (W_j^*\otimes W_j)^{\otimes m_j}.$$
   \end{cor}
   \subsection{On $K\cap H$-equivariant cohomological test vectors}\label{Subsection: construction of equivariant Cohomological vector}

   Let us retain the notation in Subsection \ref{subsection: reduction to polynomial repn}. Now we are fully ready to construct explicitly a cohomological test vector in the minimal $K$-type $\tau$ of $\pi$,
with desired properties.

Set $K' = K = \RO_{2n}(\BR)$ and $C' = H\cap K = \RO_{n}(\BR)\times \RO_n(\BR),$ where $n\geq 2$ is an integer. The maximal torus $T'$ in the previous subsection is chosen to be $(M\cap K)^0$, which is isomorphic to a product of $n$ copies of $\SO_2(\BR)$. We consider the standard representation of $K'$ on $V = \BC^{2n}$,  where every vector in $V$ is realized as a column vector and the group $K$ acts by multiplication on the left. Take a standard basis $e_1, e_2,\cdots, e_{2n}$ in $V$. Then the dual representation $V^*$ can be realized as the space of row vectors with $2n$ entries, where an element $x\in K$ acts via multiplication by $x^{-1}$ on the right. We take $e_1^t, e_2^t,\cdots e_{2n}^t$ to be the dual basis of $V^*$. For each $j = 1,2, \cdots, n$,
   \begin{equation}\label{def: uj}
           u_j := e_{2j-1}-i e_{2j},\qquad u_j^*:= e_{2j-1}^t-i e_{2j}^t
   \end{equation}
   are weight vectors of $V$ and $V^*$ respectively with weight
   $$ [0,0,\cdots,0, 1, 0,\cdots,0 ],$$
   where $1$ locates in the $j$-th position.

   For each $j=1,2,\cdots,n$, we consider the fundamental representation $\beta_j$ of $K$ on the space $W_j := \wedge^j V$. By \cite[Theorem 5.5.13]{GW09}, when $j\ne n$, the restriction of $W_j$ on $K^0$ is irreducible with highest weight
   $$[1,1,\cdots,1,0,\cdots,0],$$
   where all the $1$'s locate in the first $j$ positions; when $j = n$,  $W_n$ is only irreducible as a $K$-module (not a $K^0$-module) whose  highest weight is
   $$ [1,1, \cdots ,1].$$
   It is clear that the following linear functional on $V^{\ast}\otimes V$ is $H\cap K$-invariant:
   \begin{equation}
       \langle\lambda_{V,\textrm{endo}}, v^*\otimes v\rangle := v^*\mtrtwo{\RI_n}{}{}{0_n} v, \qquad v\in V, v^*\in V^*.
   \end{equation}
   Thus, the $H\cap K$-invariant linear functional $\lambda_{V,\textrm{endo}}$ induces an $H\cap K$-invariant linear functional $\lambda_{W_j,\textrm{endo}}$ on $W_j^*\otimes W_j$ defined as follows: for every $v_1,v_2,\cdots, v_j\in V$, $v_1^*, v_2^*, \cdots, v_j^*\in V^*$, we define
   \begin{equation}
       \begin{aligned}
       \langle\lambda_{W_j,\textrm{endo}}, (\wedge_{l=1}^j v_l)\bigotimes (\wedge_{l=1}^j v_l^*)\rangle
        := \sum_{s\in S_j} \sgn(s)\prod_{i=1}^j \langle \lambda_{V,\textrm{endo}}, v_i^*\otimes v_{s(i)}\rangle,
       \end{aligned}
   \end{equation}
   where $S_j$ is the symmetric group which permutes $j$ symbols, and $\sgn(s)$ is the sign of the permutation $s$, i.e. it is $1$ if $s$ is an even permutation, and it is $-1$ if $s$ is an odd permutation.

   We take a highest weight vector of $W_j$:
      $$\tilde{u_j}:= u_1\wedge u_2\wedge\cdots\wedge u_j$$
   and a highest weight vector of $W_j^*$:
      $$\tilde{u_j}^*:= u_1^*\wedge u_2^*\wedge\cdots\wedge u_j^*,$$
   where all $u_1, u_2,\cdots, u_j, u_1^*, u_2^*,\cdots, u_j^*$ are defined in \eqref{def: uj}.
   \begin{cor}\label{cor: fundamental even in minimal K type}
      For each $j=1,2,\cdots,n$, the following smooth function on $K$ defined by
   \begin{equation}\label{eq: even fundamental}
       f_{W_j}(x) := \langle \lambda_{W_j,\, \textrm{endo}}, \beta_j^*(x^{-1})\tilde{u_j}^*\otimes\beta_j(x^{-1})\tilde{u_j}\rangle
   \end{equation}
   is right $H\cap K$-invariant, and lives in the induced representation
   $$ W_{{\even},j} = \Ind_{(M\cap K)^0}^K \underbrace{\chi_{2}\otimes\chi_2\otimes\cdots\otimes \chi_{2}}_{j \text{ copies}}\otimes \id\otimes\cdots\otimes \id.$$ Moreover,
   $$f_{W_j}(w) = 1,$$
   where $w$ is the Weyl element defined in \eqref{eq: def of w}. Under the left $K$-action, $f_{W_j}(x)$ is a highest weight vector; and under the right $K$-action, it generates a minimal $K$-type of $W_{{\even},j}.$
   \end{cor}
   \begin{proof}
       Only the equation $f_{W_j}(w) = 1$ needs to be checked, the other statements follow directly from Proposition \ref{prop: construction of fundamental weight building function}. Let $w$ be the Weyl element in \eqref{eq: def of w}. Then
       \begin{equation}\label{eq: even fundamental}
           \begin{aligned}
             f_{W_j}(w) = &\langle \lambda_{W_j,\, \textrm{endo}}, \beta_j^*(w^{-1})\tilde{u_j}^*\otimes\beta_j(w^{-1})\tilde{u_j}\rangle\\
                        = &\sum_{s\in S_j} \sgn(s)\prod_{l=1}^j u_l^* \cdot w^{-1}\cdot\mtrtwo{\RI_n}{}{}{0_n}\cdot w\cdot  u_{s(l)}
           \end{aligned}
        \end{equation}
       By a direct matrix computation, it is easy to check that for any $1\leq l \leq j$,
       $$u_l^* \cdot w^{-1}\cdot\mtrtwo{\RI_n}{}{}{0_n}\cdot w\cdot  u_{s(l)} = \delta_{l}^{s(l)},$$
       where $\delta_{l}^k$ is the standard Kronecker delta symbol. Thus, $f_{W_j}(w)=1$.
   \end{proof}
   \begin{cor}\label{cor: weight building function minimal K type}
      Given a sequence of decreasing even integers
      \begin{equation}\label{eq: vector N}
         \vec{N'} := (N_1'\geq N_2'\geq \cdots \geq N_n' \geq 0),
      \end{equation}
      for each $j=1,2,\cdots,n$, let smooth functions $f_{W_j}(x)$ be constructed as in  \eqref{eq: even fundamental}. Then the function
      $$F_{wt, \vec{N'}}(x):= \left(\prod_{j=1}^{n-1} f_{W_j}(x)^{\frac{N_j'-N_{j+1}'}{2}}\right) \cdot f_{W_n}(x)^{\frac{N_n'}{2}}$$
      is right $K\cap H$-invariant, and lives in the induced representation $$ W_{{\even},\vec{N'}} = \Ind_{(M\cap K)^0}^K \chi_{N_1'}\otimes\chi_{N_2'}\otimes\cdots\otimes \chi_{N_n'}.$$ Moreover, we have
      $$F_{wt, \vec{N'}}(w) = 1$$
      where $w$ is the Weyl element defined in \eqref{eq: def of w}. Under the left $K$-action, $F_{wt, \vec{N'}}(x)$ is a highest weight vector; and under the right $K$-action, it generates a minimal $K$-type of $W_{{\even}, \vec{N'}}.$
   \end{cor}
   \begin{proof}
      The fact that $F_{wt, \vec{N'}}(w) = 1$ follows directly from Corollary \ref{cor: fundamental even in minimal K type}. Thus, Corollary \ref{cor: weight building function minimal K type} follows directly from Corollary \ref{cor: weight building function in general}.
   \end{proof}

   Recall that $\chi$ and $\omega$ are characters of $\BR^\times$, $\chi_0$ and $\omega_0$ are the restrictions of the characters $\chi$ and $\omega$ on $\{\pm 1\}$. Our goal is to construct a $\BC$-valued function
in the minimal $K$-type $\tau_0$ of $\pi_{\textrm{sc}}$ (defined in \eqref{eq: induced space complex value}) satisfying the right $K\cap H$-equivariance as in \eqref{eq: right equivariant property}.
In the special case $\omega_0 = \chi_0 = \id$, we set $N_j' = l_j+1$ ($j=1,2,\cdots,n$). Then $\vec{N'} = (N_1',N_2', \cdots ,N_n')$ is a decreasing sequence of positive integers.
By Corollary \ref{cor: weight building function minimal K type}, the function $F_{wt, \vec{N'}}(x)$ constructed therein is exactly the right $K\cap H$-invariant vector in the minimal $K$-type $\tau_0$ of
$\pi_{\textrm{sc}}$.
To take care of the other cases, i.e. one of $\chi_0, \omega_0$ is non-trivial, we introduce two more functions.

   \begin{prop}\label{prop: det lives in fundamental repn}
           Let $u_j^*$ be as in \eqref{def: uj} and  the Weyl element $w$ be defined in \eqref{eq: def of w}. Define the following two functions on $K$:
           \begin{equation}\label{eq: det polynomial 1}
               F_{\lt}(x) := \det\left(
                                        \left(\begin{array}{c}
                                          u_1^* \\
                                          u_2^* \\
                                          \cdots \\
                                          u_n^* \\
                                        \end{array}\right)\cdot x \cdot \left(
                                                                   \begin{array}{c}
                                                                     \RI_n \\
                                                                     0_n \\
                                                                   \end{array}
                                                                 \right)
                                      \right)
           \end{equation}
           and
           \begin{equation}\label{eq: det polynomial 2}
               F_{\rt}(x) := i^n\cdot\det\left(
                                        \left(\begin{array}{c}
                                          u_1^* \\
                                          u_2^* \\
                                          \cdots \\
                                          u_n^* \\
                                        \end{array}\right)\cdot x \cdot \left(
                                                                   \begin{array}{c}
                                                                     0_n \\
                                                                     \RI_n \\
                                                                   \end{array}
                                                                 \right)
                                      \right)
           \end{equation}
           Then $F_{\lt}(w) = F_{\rt}(w)=1$. Under the left $K$-action, both $F_{\lt}(x)$ and $F_{\rt}(x)$ are highest weight vectors, while under the right $K$-action,
they both generate the minimal $K$-type of the induced representation
           $$W_{\textrm{odd}} := \Ind_{(M\cap K)^0}^K \chi_1\otimes\chi_1\otimes\cdots\otimes\chi_1.$$
       \end{prop}

       \begin{proof}
           By a straightforward matrix computation, one can check that $F_{\lt}(w) = F_{\rt}(w)=1$ and that both $F_{\lt}(x)$ and $F_{\rt}(x)$ live in the induced representation $W_{\textrm{odd}}$.
It remains to show that $F_{\lt}(x)$ generates the minimal $K$-type of $W_{\textrm{odd}}$. The proof for $F_{\rt}(x)$ is exactly the same.

           Note that $u_1^*\wedge u_2^* \wedge \cdots \wedge u_n^*$ is a highest weight vector in the fundamental representation $\bigwedge^n V^*$. Under the left $K$-action, $F_{\lt}(x)$ is a highest vector and
generate an irreducible submodule of $C^\infty(K')$ isomorphic to $\bigwedge^n V^*$. By Lemma \ref{lemma: useful dual lemma}, under the right $K$-action, $F_{\lt}(x)$ generates an irreducible $K$-submodule of $C^\infty(K')$ isomorphic to $\bigwedge^n V$. Since the representation $\bigwedge^n V$ has highest weight $[1,1, \cdots ,1,1]$, it is isomorphic to the minimal $K$-type of $W_{\textrm{odd}}$. We are done.
       \end{proof}

       \begin{defn}\label{def: parity compatible}
           Let $\vec{N}:=(N_1,N_2,\cdots,N_n)$ be a sequence of positive integers with the same parity in the decreasing order and $\omega$ be a character of $\BZ/2\BZ$. We say $\vec{N}$ is parity-compatible with $\omega_0$ if when $N_1$ is even, $\omega_0$ is trivial; and when $N_1$ is odd, $\omega_0$ is the sign character.
       \end{defn}
    For any decreasing sequence $\vec{N'}:=(N_1',N_2',\cdots,N_n')$ of non-negative even integers, recall the function $F_{wt, \vec{N'}}(x)$ as constructed in Corollary \ref{cor: weight building function minimal K type}
and the functions $F_{\lt}(x)$, $F_{\rt}(x)$ as constructed in Proposition \ref{prop: det lives in fundamental repn}. Now we are ready to construct the right $(K\cap H)$-equivariant cohomological test vector in the induced representation $\pi_{\textrm{sc}}$ defined in \eqref{eq: induced space complex value}, based on these three types of functions.
    \begin{thm}\label{thm: phi0 lives in minimal K type}
        Let $\vec{N}:=(N_1,N_2,\cdots,N_n)$ be a sequence of positive integers with the same parity in the decreasing order, and $\omega_0$ be a character of $\BZ/2\BZ$ that is parity-compatible
 with $\vec{N}$, as in Definition \ref{def: parity compatible}. For any character $\chi_0$ of $\BZ/2\BZ$, define a function $\varphi_{\vec{N},\chi_0\otimes\chi_0\omega_0}(x)\in C^\infty(K)$ as follows:
        \begin{enumerate}
          \item If $N_1$ is even and $\chi_0$ is trivial, set
                 $\varphi_{\vec{N},\chi_0\otimes\chi_0\omega_0}(x) := F_{wt, \vec{N'}}(x),$
                 where $N_j' = N_j$, $j=1,2,\cdots,n$.
          \item If $N_1$ is even and $\chi_0$ is the sign character, set
                 $$ \varphi_{\vec{N},\chi_0\otimes\chi_0\omega_0}(x) := F_{wt, \vec{N'}}(x)F_{\lt}(x)F_{\rt}(x),$$
                where $N_j' = N_j-2$, $j=1,2,\cdots,n$.
          \item If $N_1$ is odd and $\chi_0$ is trivial, set
                  $$ \varphi_{\vec{N},\chi_0\otimes\chi_0\omega_0}(x) := F_{wt, \vec{N'}}(x)F_{\rt}(x),$$
                where $N_j' = N_j-1$, $j=1,2,\cdots,n$.
          \item If $N_1$ is odd and $\chi_0$ is the sign character, set
                  $$ \varphi_{\vec{N},\chi_0\otimes\chi_0\omega_0}(x) := F_{wt, \vec{N'}}(x)F_{\lt}(x),$$
                where $N_j' = N_j-1$, $j=1,2,\cdots,n$.
        \end{enumerate}
       Then $\varphi_{\vec{N},\chi_0\otimes\chi_0\omega_0}(x)$ lives in the minimal $K$-type of
       $$\Ind_{(K\cap M)^0}^K \chi_{N_1}\otimes\chi_{N_2}\otimes\cdots\otimes\chi_{N_2}$$
       and satisfies the right $K\cap H$-equivariant property \eqref{eq: right equivariant property}. Moreover, $\varphi_{\vec{N},\chi_0\otimes\chi_0\omega_0}(w) = 1$, where $w$ is the Weyl element defined in \eqref{eq: def of w}.
    \end{thm}
    \begin{proof}
       The fact that $\varphi_{\vec{N},\chi_0\otimes\chi_0\omega_0}(x)$ lives in
       $$\Ind_{(K\cap M)^0}^K \chi_{N_1}\otimes\chi_{N_2}\otimes\cdots\otimes\chi_{N_2}$$ and $\varphi_{\vec{N},\chi_0\otimes\chi_0\omega_0}(w) = 1$ follows directly from Corollary \ref{cor: weight building function minimal K type} and Proposition \ref{prop: det lives in fundamental repn}. The right $(K\cap H)$-equivariance can be checked in an ad-hoc way by direct matrix computations, since $F_{wt, \vec{N'}}(x)$ is right $K\cap H$-invariant. Here we only prove that $\varphi_{\vec{N},\chi_0\otimes\chi_0\omega_0}(x)$ generates the minimal $K$-type of $$\Ind_{(K\cap M)^0}^K \chi_{N_1}\otimes\chi_{N_2}\otimes\cdots\otimes\chi_{N_2}.$$

       Note that $F_{wt, \vec{N'}}(x)$ constructed in Corollary \ref{cor: weight building function minimal K type} and the functions $F_{\lt}(x)$, $F_{\rt}(x)$ constructed in Proposition \ref{prop: det lives in fundamental repn} are all highest weight vectors under the left $K$-action. Under the left $K$-action, the function $\varphi_{\vec{N},\chi_0\otimes\chi_0\omega_0}(x)$ is also a highest weight vector and generates an irreducible $K$-submodule of $C^\infty(K)$ with highest weight $(N_1, N_2, \cdots, N_n).$ The irreducible $K$-module with highest weight $(N_1, N_2, \cdots, N_n)$ is self-dual. By Lemma \ref{lemma: useful dual lemma}, under the right $K$-action, $\varphi_{\vec{N},\chi_0\otimes\chi_0\omega_0}(x)$ generates an irreducible $K$-submodule with highest weight $(N_1, N_2, \cdots, N_n)$, which is exactly the minimal $K$-type of the induced representation $$\Ind_{(K\cap M)^0}^K \chi_{N_1}\otimes\chi_{N_2}\otimes\cdots\otimes\chi_{N_n}.$$
    \end{proof}
    It is clear that once we set $N_j := l_j+1$, the function $\varphi_{\vec{N},\chi_0\otimes\chi_0\omega_0}(x)$ constructed in the above theorem must live in the minimal $K$-type of the induced representation
$\pi_{\textrm{sc}}$
as defined in \eqref{eq: induced space complex value}. Hence, combining with the reduction steps in Subsection \ref{subsection: reduction to polynomial repn}, we obtain the following direct corollary of Theorem \ref{thm: phi0 lives in minimal K type}.
    \begin{cor}\label{cor: construction of cohomological vector}
        Let $\chi_0,\omega_0$ be the restriction of $\chi,\omega$ on $\{\pm 1\}$ as in \eqref{eq: right equivariant property}. Let $N_j := l_j+1$ and $\varphi_{\vec{N},\chi_0\otimes\chi_0\omega_0}$ be the function constructed in Theorem \ref{thm: phi0 lives in minimal K type}. Define $f_{1,2,\cdots,n} = \varphi_{\vec{N},\chi_0\otimes\chi_0\omega_0}$ and $f_{\eta_1,\eta_2,\cdots,\eta_n}(k)$ by the rules \eqref{Eq: component inv}.
Then, as defined in \eqref{Eq: basis expansion of phi}, the function
        $$
        f(k) = \sum f_{\eta_1,\eta_2,\cdots,\eta_n}(k) v_{\eta_1}\otimes v_{\eta_2}\otimes\cdots\otimes v_{\eta_n}
        $$
    with the summation being taken over all possible choices $\eta_j\in \{\pm j\}$,
    belongs to the minimal $K$-type $\tau$ of $\pi$. Moreover, $f(k)$ satisfies the desired equivariant property \eqref{eq: right equivariant property}.
    \end{cor}
    \begin{rk}
        By the above construction, $f(k)$ depends on $\chi$ and $\omega$.
    \end{rk}


\section{Non-vanishing of Archimedean Local Integrals}\label{sec-NV-ALI}


With the explicit construction of the cohomological vector (see Subsection \ref{Subsection: construction of equivariant Cohomological vector}), we can analyze the local integrals $Z(v,s,\chi)$ (defined in \eqref{Eq: Def of Local Integral}) and $\Lambda_{s,\chi}(v)$ (defined in \eqref{eq: new def H inv linear functional}) in detail. We retain all notation in Subsection  \ref{Subsection: construction of equivariant Cohomological vector}.

\begin{thm}\label{thm: cohomological test vector}
    Let $f(k)$ be the cohomological vector constructed in Corollary \ref{cor: construction of cohomological vector} and $w$ be the Weyl element defined in \eqref{eq: def of w}. Then
    \begin{equation*}
       f(w) = \otimes_{j=1}^n (v_j+\chi(-1)v_{-j}).
    \end{equation*}
    As a consequence, as a meromorphic function of $s\in \BC$, $\frac{1}{L(s,\pi\otimes\chi)}\Lambda_{s,\chi}(f) = 1.$ Hence $\Lambda_{s,\chi}$ defines a non-zero element in
    \begin{equation*}
       \Hom_{H}(\pi, \abs{\det}^{-s+\frac{1}{2}}\chi^{-1}(\det)\otimes\abs{\det}^{s-\frac{1}{2}}(\chi\omega)(\det)),
    \end{equation*}
    whenever $s$ is not a pole of $L(s, \pi\otimes\chi).$
\end{thm}
\begin{proof}
     By Theorem \ref{thm: phi0 lives in minimal K type}, we have that $\varphi_{\vec{N},\chi_0\otimes\chi_0\omega_0}(w) = 1$, and hence
     \begin{equation*}
        f_{1,2,\cdots,n}(w) = 1.
     \end{equation*}
     Recall the matrix $c(\epsilon_1,\cdots\epsilon_n)$ from \eqref{Eq: parametrize M intersect K 2}. By \eqref{Eq: component inv},
     \begin{eqnarray*}
         f_{\eta_1,\eta_2,\cdots,\eta_n}(w) &=& f_{1,2,\cdots,n}(c(\sgn(\eta_1),\sgn(\eta_2),\cdots,\sgn(\eta_n))w)\\
         & =& \prod_{j=1}^n \chi(\sgn( \eta_j)).
     \end{eqnarray*}
     It follows that by \eqref{Eq: basis expansion of phi}
     \begin{equation*}
        \begin{aligned}
       f(w) &= \sum_{\text{all }\eta_j\in\{\pm j\}} f_{\eta_1,\eta_2,\cdots,\eta_n}(w) v_{\eta_1}\otimes v_{\eta_2}\otimes\cdots\otimes v_{\eta_n}\\
       &= \sum_{\text{all }\eta_j\in\{\pm j\}}\prod_{j=1}^n \chi(\sgn(\eta_j)) v_{\eta_1}\otimes v_{\eta_2}\otimes\cdots\otimes v_{\eta_n}\\
       &= \otimes_{j=1}^n (v_j+\chi(-1)v_{-j}).
       \end{aligned}
    \end{equation*}
     By Corollary \ref{cor: construction of cohomological vector}, $f(k)$ satisfies the desired equivariant property \eqref{eq: right equivariant property}. As a direct consequence of \eqref{eq: 020} and Corollary \ref{Cor: GL(2, R) Non-vanishing}, the following holds for $s\in \BC$ with sufficiently large real parts:
     \begin{equation*}
         \begin{aligned}
         \Lambda_{s,\chi}(f) &= \langle \bigotimes_{i=1}^n \lambda_i, f(w) \rangle\\
                             &= \langle \bigotimes_{i=1}^n \lambda_i, \otimes_{j=1}^n (v_j+\chi(-1)v_{-j}) \rangle\\
                             &= \prod_{j=1}^m L(s,D_{l_j}\abs{\quad}^\frac{m}{2}\otimes\chi)\\
                             &= L(s,\pi\otimes\chi).
         \end{aligned}
     \end{equation*}
     Here the local $L$-function is obtained from the local Langlands correspondence for $\GL_{2n}(\BR)$, which we refer to \cite{L89} (also see \cite{Kn94} and \cite{J}). Thus, by meromorphic continuation, for all $s\in \BC,$ $\frac{1}{L(s,\pi\otimes\chi)}\Lambda_{s,\chi}(f) = 1.$ We are done.
\end{proof}
     It is finally clear that Theorem \ref{thm-main} holds. As a consequence, we can also show the non-vanishing of the archimedean Friedberg-Jacquet integral $Z(v,s,\chi)$.
\begin{cor}\label{cor: relation between two linear models}
     There exists a holomorphic function $G(s,\chi)$ such that
     \begin{equation*}
         Z(v,s,\chi) = e^{G(s,\chi)}\Lambda_{s,\chi}(v).
     \end{equation*}
     As a consequence, if $v=f(k)$ is the cohomological vector constructed in Corollary \ref{cor: construction of cohomological vector}, then whenever $s$ is not a pole of $L(s, \pi\otimes\chi)$, $Z(v,s,\chi)$ does not vanish.
\end{cor}
\begin{proof}
     By \cite[Theorem B]{Ch-Sun}, for all but countably many $s\in \BC$ where $L(s,\pi\otimes\chi)$ does not have a pole, one has
     \begin{equation*}
        \text{dim Hom}_{H}(\pi, \abs{\det}^{-s+\frac{1}{2}}\chi^{-1}(\det)\otimes\abs{\det}^{s-\frac{1}{2}}(\chi\omega)(\det))\leq 1.
    \end{equation*}
    Since for such pair $(s,\chi)$, both $Z(v,s,\chi)$ and $\Lambda_{s,\chi}$ defines a non-zero element in
    \begin{equation*}
       \text{Hom}_{H}(\pi, \abs{\det}^{-s+\frac{1}{2}}\chi^{-1}(\det)\otimes\abs{\det}^{s-\frac{1}{2}}(\chi\omega)(\det)),
    \end{equation*}
    there exists a constant $C(s,\chi)$ (depending on $s$ and $\chi$) such that
    \begin{equation}\label{eq: 021}
       Z(v,s,\chi) = C(s,\chi)\Lambda_{s,\chi}(v).
    \end{equation}
    Since both $Z(v,s,\chi)$ and $\Lambda_{s,\chi}$ are meromorphic in $s$ and $\chi$, $C(s,\chi)$ is also meromorphic in
    $s,\chi$. Now we plug $v=f(k)$ (the cohomological vector constructed in Corollary \ref{cor: construction of cohomological vector}) in \eqref{eq: 021},
    and by Theorem \ref{thm: cohomological test vector}, we obtain
    \begin{equation*}
       C(s,\chi) = \frac{Z(f,s,\chi)}{\Lambda_{s,\chi}(f)} = \frac{Z(f,s,\chi)}{L(s,\pi\otimes\chi)}.
    \end{equation*}
    Thus, by \cite[Theorem 3.1]{A-G-J}, $C(s,\chi)$ must be holomorphic. Similarly, also by \cite[Theorem 3.1]{A-G-J}, we can choose a smooth vector $v_0$ such that $Z(v_0,s,\chi)=L(s,\pi\otimes\chi)$. Thus, using the same argument as above,
    \begin{equation*}
       \frac{1}{C(s,\chi)} = \frac{\Lambda_{s,\chi}(v_0)}{L(s,\pi\otimes\chi)}
    \end{equation*}
    must also be holomorphic, by Corollary \ref{Cor: analytic property of Lambda}. Hence $C(s,\chi)$ have no zeroes. This implies that there exists a holomorphic function $G(s,\chi)$ such that
         $C(s,\chi) = e^{G(s,\chi)}$.
\end{proof}

\bibliographystyle{elsarticle-num}

\end{document}